\documentclass{article}

\usepackage[T1]{fontenc}
\usepackage[utf8]{inputenc}

\usepackage[margin=1.5in]{geometry}
\usepackage{graphicx}

\usepackage[shortlabels]{enumitem}

\usepackage{mathtools,amssymb,amsthm}
\numberwithin{equation}{section}

\usepackage[
natbib,
maxcitenames=3, 
maxbibnames=10,  
sorting=nyt,
url=false,
sortcites,
defernumbers,
doi=false,
backref,
backend=biber
]{biblatex}
\usepackage{hyperref}

\usepackage{todonotes}
\usepackage{url}

\usepackage[nameinlink, capitalise, noabbrev]{cleveref}

\usepackage{subcaption}
\usepackage{xfrac}
\usepackage{nicefrac}
\usepackage{tikz}
\usepackage{pgfplots}
\usetikzlibrary{3d,calc}

\newtheorem{theorem}{Theorem}
\newtheorem{proposition}{Proposition}[section]
\newtheorem{lemma}{Lemma}[section]
\newtheorem{corollary}{Corollary}[section]
\theoremstyle{remark}
\newtheorem{remark}{Remark}[section]

\theoremstyle{definition}

\newtheorem{definition}{Definition}[section]

\numberwithin{equation}{section}

\newcommand{\N}{\mathbb{N}}

\newcommand{\R}{\mathbb{R}}

\newcommand{\norm}[1]{\left\Vert #1 \right\Vert}
\newcommand{\abs}[1]{\left\vert #1 \right\vert}

\newcommand{\st}{\,:\,}
\DeclareMathOperator{\supp}{supp}
\newcommand{\dx}{\,\mathrm{d}x}
\renewcommand{\d}{\,\mathrm{d}}

\newcommand{\eps}{\varepsilon}
\DeclareMathOperator{\dist}{dist}
\DeclareMathOperator{\Lip}{Lip}

\newcommand{\grad}{\nabla}

\newcommand{\one}{\chi}

\pgfplotsset{compat=1.17}

\begin{document}

\title{Convergence rates of the fractional to the local Dirichlet problem}

\author{Leon Bungert%
\thanks{%
Institute of Mathematics, Center of Artifical Intelligence and Data Science (CAIDAS), University of Würzburg, Emil-Fischer-Str. 40, 97074 Würzburg, Germany. \href{mailto:leon.bungert@uni-wuerzburg.de}{leon.bungert@uni-wuerzburg.de}
}
\and 
Félix del Teso%
\thanks{Departamento de Matemáticas, Universidad Autónoma de Madrid, Campus de Cantoblanco, 28049 Madrid, Spain. \href{mailto:felix.delteso@uam.es}{felix.delteso@uam.es}
}
}

\maketitle

\begin{abstract}
    We prove non-asymptotic rates of convergence in the $W^{s,2}(\R^d)$-norm for the solution of the fractional Dirichlet problem to the solution of the local Dirichlet problem as $s\uparrow 1$. 
    For regular enough boundary values we get a rate of order $\sqrt{1-s}$, while for less regular data the rate is of order $\sqrt{(1-s)\abs{\log(1-s)}}$. We also obtain results when the right hand side depends on $s$, and our error estimates are true for all $s\in(0,1)$.
    The proofs use variational arguments to deduce rates in the fractional Sobolev norm from energy estimates between the fractional and the standard Dirichlet energy.
\end{abstract}



\section{Introduction}

In this paper we are concerned with the Dirichlet problem for the fractional Laplace operator with fractional power $s\in(0,1)$ and we will study at which rate it converges to the Dirichlet problem for the standard local Laplace operator as $s$ tends to $1$. 
To set the scene let us consider, for $s\in(0,1)$, the fractional Dirichlet problem
\begin{align}\label{eq:fld}
    \begin{dcases}
        (-\Delta)^s u_s = f_s,\qquad&\text{in }\Omega,\\
        u_s = g,\qquad&\text{on }\R^d\setminus\Omega,
    \end{dcases}    
\end{align}
where $f_s$ and $g$ are given functions and $\Omega$ is a bounded Lipschitz domain in $\R^d$. 
In our convention, the fractional Laplacian $(-\Delta)^s$ of a smooth function $\phi$ is defined as the hypersingular integral
\begin{align*}
    (-\Delta)^s \phi(x) := 4C_{s,d}\operatorname{P.V.}\int_{|x-y|>0}
    \frac{\phi(x)-\phi(y)}{\abs{x-y}^{d+2s}}\d y, \quad \textup{with} \quad C_{s,d}:=\frac{d}{\omega_d}(1-s),
\end{align*}
where $\omega_d$ denotes the $(d-1)$-dimensional Hausdorff measure of $\partial B_1\subset\R^d$, the unit sphere in $\R^d$.
The choice of our normalization constant $C_{s,d}$ is for convenience in the proofs and the more standard normalization constant can be used just as well (see \cref{sec:normalization} for a brief discussion).
The aim of this paper is to obtain non-asymptotic error estimates between $u_s$ and the solution of the local Dirichlet problem
\begin{align}\label{eq:ld}
    \begin{dcases}
        -\Delta u = f,\qquad&\text{in }\Omega,\\
        u = g,\qquad&\text{on }\partial\Omega,
    \end{dcases}    
\end{align}
which arises as a limit of \labelcref{eq:fld} as $s\uparrow 1$.
The fractional Laplacian and fractional Sobolev spaces have been extensively investigated in the literature and we refer to \cite{di2012hitchhikers,leoni2023first,garofalo2018fractionalthoughts} for an overview of the field.

A key result which connects fractional with standard Sobolev spaces is the theory by Bourgain--Brezis--Mironescu~\cite{Brezis2001Sobolev}, who proved that fractional Sobolev semi-norms converge to standard ones as $s\uparrow 1$. See also \cite{Fog23} for generalizations and \cite{Dav02} for a BV version.
While this does not yet allow to prove that minimizers of energies involving the fractional Sobolev semi-norm converge to corresponding minimizers as $s\uparrow 1$, in \cite{ponce2004new} Gamma-convergence of the semi-norms was proved. 
See also \cite{FoKaVo20} for a general Mosco-convergence result.
Using this or, alternatively, relying on PDE tools as done in \cite{biccari2018poisson} one can prove that solutions of \labelcref{eq:fld} converge to the solution of \labelcref{eq:ld}. 
Similar results were established for a fractional porous medium equation in \cite{delTeso2017PME}, for a nonlocal divergence theorem in \cite{kassmann2024divergencethm}, for Neumann problems related to the fractional Laplacian in \cite{GrHe24}, for parabolic problems with a nonlocal Laplacian in \cite{AnMaRoTo10}, for nonlocal quasilinear equations in \cite{ChJa17,BjCaFi12}, for fractional eigenvalues and eigenfunctions in \cite{BrPaSq16,FeSa20,Fogthesis}, and for a nonlocal Cahn--Hilliard equation in \cite{AbHu24}.
Notably, the last paper even proved a convergence rate using a sharp consistency result of the associated nonlocal (not fractional!) Laplacian, which is in the same spirit as existing results on the graph Laplacian on random geometric graphs \cite{calder2020calculus,calder2023rates}.
Asymptotic expansions of the fractional $p$-Laplacian and their local limits were investigated in \cite{BuSq22,dTMeOc23,dTEnLe22} and stable cones for the fractional perimeter with $s$ close to $1$ were studied in \cite{CaCiSe20}.

The \textbf{purpose of this paper is to prove rates of convergence} of the solutions of \labelcref{eq:fld} to the solution of \labelcref{eq:ld} in a strong Sobolev norm and in a non-asymptotic way---meaning the rates are valid for all $s\in(0,1)$ and not just in the limit $s\uparrow 1$.
As we explain in \cref{sec:outlook}, our framework, which relies on a quantitative variational approach, also applies to other fractional problems. 
The approach works by taking solutions of \labelcref{eq:fld,eq:ld} and using them to construct feasible competitors of the other variational problem, respectively.
These competitors will be constructed in such a way that their energies are under good control, which will allow us to derive the final rates.
Since the solution of the local problem \labelcref{eq:ld} is automatically feasible for the nonlocal one \labelcref{eq:fld} (in the sense that it belongs to the correct functional space of the energy minimization problem), the challenging part is to build a competitor for the local problem. For this we shall use a tailored mollification procedure which maps $W^{s,2}(\R^d)$-functions into $W^{1,2}(\R^d)$-functions with a quantitative control of the latter norm. 
Furthermore, since such a mollification alters the boundary values of a function we have to perform adaptions close to the boundary which require careful tail estimate for the used mollification kernel which has only algebraic decay.

Related techniques have been used for discrete to continuum limits of linear and nonlinear PDEs on graphs in a series of works~\cite{burago2015graph,garcia2020error,garcia2022graph,trillos2018variational,calder2022improved,calder2020calculus,bungert2024poissonlearning}.
We also refer to \cite{kreisbeck2024nonconstant} and the references therein for a related framework, albeit without rates.

The final rates we obtain are  \[\norm{u_s-u}_{W^{s,2}(\R^d)}^2\lesssim{1-s+\norm{f-f_s}_{L^1(\Omega)}}\]  for a regular enough boundary value $g$ (see \cref{thm:main,cor:main_regular_bdry}) and
\[\norm{u_s-u}_{W^{s,2}(\R^d)}^2\lesssim{(1-s)|\log(1-s)|+\norm{f-f_s}_{L^1(\Omega)}}\] for less regular ones  (see \cref{thm:less_regular_bdry}).
To the best of our knowledge, the only other rigorous convergence result with a rate can be found in \cite{saldana2023differentiability}, where the authors proved that asymptotically as $s\uparrow 1$ it holds $\norm{u_s-u}_{L^\infty(\Omega)}\lesssim 1-s$.
For this they used that for regular enough right hand sides $f_s=f$ and for $g\equiv 0$ the map $s\mapsto u_s$ is continuously differentiable in $s$ (see \cite{jarohs2020new}).
Note that this result does not imply that our rate is not sharp, since we obtain a rate in the $W^{s,2}(\R^d)$-norm, which is not comparable to the $L^\infty(\R^d)$-norm. 
The $W^{s,2}(\R^d)$-convergence rate which we prove was conjectured but not proved in \cite{biccari2018poisson}.
Furthermore, the numerical examples in this paper indicate the rate is sharp and future work will focus on proving this.

\subsection{Notation}

A function $\phi:\R^d\to\R$ is called $\beta$-Hölder for $\beta\in(0,1]$ if its Hölder-seminorm
\begin{align*}
    [\phi]_{C^{0,\beta}(\R^d)} := \sup_{\substack{x,y\in\R^d\\x\neq y}}\frac{\abs{\phi(x)-\phi(y)}}{\abs{x-y}^\beta}
\end{align*}
is finite.
If $\phi$ is also bounded, we write $\phi\in C^{0,\beta}(\R^d)$.
Analogously, we denote higher-order Hölder spaces by $C^{k,\beta}(\R^d)$ for $k\in\N_0$ and higher-order Hölder seminorms by
\begin{align*}
    [\phi]_{C^{k,\beta}(\R^d)}
    :=
    \max_{\abs{\gamma}=k}
    [D^\gamma\phi]_{C^{0,\beta}(\R^d)},
\end{align*}
where $\gamma\in\N^d$ denotes a multi-index and $D^\gamma\phi:=\partial_{\gamma_1}\dots\partial_{\gamma_d}\phi$.

Let us also define, for $s\in(0,1)$, the $W^{s,2}(\R^d)$-seminorm
\begin{align*}
    [\phi]_{W^{s,2}(\R^d)}^2
    := 
    \frac{2d}{\omega_d}
    (1-s)
    \int_{\R^d}
    \int_{\R^d}
    \frac{\abs{\phi(x)-\phi(y)}^2}{\abs{x-y}^{d+2s}}
    \d y \d x,
\end{align*}
which is scaled in such a way that 
\begin{align*}
    \lim_{s\uparrow 1}[\phi]_{W^{s,2}(\R^d)}^2 = \int_\Omega\abs{\nabla\phi(x)}^2\d x=:[\phi]_{W^{1,2}(\R^d)}^2,\qquad
    \forall \phi\in W^{1,2}(\R^d).
\end{align*}
A function $\phi$ is said to belong to $W^{s,2}(\R^d)$ for $s\in(0,1]$ if the following norm is finite,
\begin{align*}
    \norm{\phi}_{W^{s,2}(\R^d)}^2 := \norm{\phi}_{L^2(\R^d)}^2 + [\phi]_{W^{s,2}(\R^d)}^2.
\end{align*}
For prescribing the right decay for the boundary datum $g$ at infinity we also define (potentially in conflict with other conventions in the literature and omitting any normalization constants) the following homogeneous vector-valued fractional space 
\begin{align*}
    \dot W^{s,1}(\R^d) := \left\{\phi : \R^d\to\R^d\text{ measurable}\st[\phi]_{\dot W^{s,1}(\R^d)}:= \int_{\R^d}\int_{B_1(x)}\frac{\abs{\phi(x)-\phi(y)}}{\abs{x-y}^{d+s}}\d y\d x<\infty\right\}.
\end{align*}

\subsection{Main results}
\label{sec:main_results}

Our tacit assumption for the rest of the paper is that $\Omega$ is a bounded Lipschitz domain such that we can use the Poincaré inequality.
The first main result is the following convergence rate in the $W^{s,2}(\R^d)$-norm.

\begin{theorem}[$W^{s,2}(\R^d)$-convergence rate with zero boundary condition]\label{thm:main}
Assume $s\in(0,1)$, $f_s,f\in L^\infty(\Omega)$ and $g\equiv 0$ in $\R^d\setminus\Omega$. Let $u_s\in W^{s,2}(\R^d)$ and $u\in W^{1,2}(\R^d)$ be the solutions of \labelcref{eq:fld,eq:ld} respectively. Then, there exists a constant $C=C(\Omega,d,u,u_s,f,f_s)$ such that
\begin{align*}
    \norm{u-u_s}_{W^{s,2}(\R^d)}^2 \leq  C \left(1-s+\norm{f_s-f}_{L^1(\Omega)}\right).
\end{align*}
More precisely, the behaviour of $C$ is
\[
C= K(\Omega,d,u,f)\left(
\|u_s\|_{C^{0,s}(\R^d)}+\|u_s\|_{C^{0,s}(\R^d)}^2 + \|f_s\|_{L^\infty(\R^d)} \right).
\]
\end{theorem}
\begin{remark}[Uniformly bounded Hölder norms]\label{rem:hölder} 
We want to stress that, if the domain $\Omega$ satisfies an exterior ball condition and, hence, has no inward corners, the solution $u_s$ of the nonlocal problem belongs to the space $C^{0,s}(\R^d)$ since $f_s\in L^\infty(\R^d)$ by the classical result of Ros-Oton and Serra \cite{ros-oton2014dirichlet}. 
Moreover, the constant $C$ in the above theorem  can be bounded independently of $\|u_s\|_{C^{0,s}(\R^d)}$.
This follows from the revisit of the results in \cite{ros-oton2014dirichlet} given in Lemma 3.4 (ii) in \cite{jarohs2020new}, which states that
\[
\|u_s\|_{C^{0,s}(\R^d)} \leq \tilde{C} \|f_s\|_{L^\infty(\R^d)} 
\]
for a constant $\tilde{C}$ depending only on the domain $\Omega$.
\end{remark}

\cref{thm:main}, which is proved in \cref{sec:proof_main}, has an important corollary that we state in the following.
We note that the assumption $g\equiv 0$ in $\R^d\setminus\Omega$ is no big restriction since one can incorporate non-zero and sufficiently smooth boundary conditions $g$ into the right hand side $f_s$ and get the same convergence rate.

\begin{corollary}[$W^{s,2}(\R^d)$-convergence rate with smooth boundary conditions] \label{cor:main_regular_bdry}
Assume $s\in(0,1)$, $f_s,f\in L^\infty(\Omega)$ and $g\in C^{2,\alpha}(\R^d)$ for some $\alpha>0$. Let $u_s\in W^{s,2}(\R^d)$ and $u\in W^{1,2}(\R^d)$ be the solutions of \labelcref{eq:fld,eq:ld} respectively. Then, there exists a constant $C=C(\Omega,d,u,u_s,f,f_s,g)$ such that
\begin{align*}
    \norm{u-u_s}_{W^{s,2}(\R^d)}^2 \leq  C \left(1-s+\norm{f_s-f}_{L^1(\Omega)}\right).
\end{align*}
More precisely, the behaviour of $C$ is
\[
C= K(\Omega,d,u,f)\left(\|u_s\|_{C^{0,s}(\R^d)}+\|u_s\|_{C^{0,s}(\R^d)}^2 + \|f_s\|_{L^\infty(\R^d)} + \alpha^{-1}\|g\|_{C^{2,\alpha}(\R^d)} \right).
\]
\end{corollary}
\begin{remark}\label{rem:hölder2}
    Under the same conditions as in \cref{rem:hölder} the uniform bound $\norm{u_s}_{C^{0,s}(\R^d)} \lesssim (\|f_s\|_{L^\infty(\R^d)} + \alpha^{-1}\|g\|_{C^{2,\alpha}(\R^d)})$ holds.
\end{remark}
If the boundary data $g$ is less regular, we can also get a convergence rate, which is slightly slower by a logarithmic factor. This is not a direct consequence of \cref{thm:main}, and needs some nontrivial adaptations. 
\begin{theorem}[$W^{s,2}(\R^d)$-convergence rate with less regular boundary conditions]\label{thm:less_regular_bdry}
    Assume $s\in(0,1)$, $f_s,f\in L^\infty(\Omega)$ and $g\in C^{1,\alpha}(\R^d)\cap W^{s,2}(\R^d)$ and $\nabla g\in\dot W^{\beta,1}(\R^d)$ for some $\alpha,\beta\in(0,1]$. Let $u_s\in W^{s,2}(\R^d)\cap C^{0,s}(\R^d)$ and $u\in W^{1,2}(\R^d)$ be the solutions of \labelcref{eq:fld,eq:ld} respectively. Then, there exists a constant $C=C(\Omega,d,u,u_s,f,f_s,g)$ such that
\begin{align*}
    \norm{u-u_s}_{W^{s,2}(\R^d)}^2 \leq  C \left((1-s)\abs{\log\left(1-s\right)}+\norm{f_s-f}_{L^1(\Omega)}\right).
\end{align*}
More precisely, the behaviour of $C$ is
\begin{align*}
    C&= K(\Omega,d,u,f)\Big(\|u_s\|_{C^{0,s}(\R^d)}+\|u_s\|_{C^{0,s}(\R^d)}^2 + \|f_s\|_{L^\infty(\R^d)} \\
    &\qquad\qquad\qquad\qquad + \alpha^{-1}\|g\|_{C^{1,\alpha}(\R^d)}+\norm{g}_{C^{0,1}(\R^d)}[g]_{\dot W^{\beta,1}(\R^d)} \Big).
\end{align*}
\end{theorem}
\begin{remark}\label{rem:hoelder?}
    Although we expect that uniform bounds, similar to those in \cite{jarohs2020new}, of the Hölder constants of solutions to \labelcref{eq:fld} with non-zero boundary data $g\in C^{1,\alpha}(\R^d)$ or even less regularity are true, we are not aware of any results in the literature which assert these.
\end{remark}
The proofs of \cref{thm:main,thm:less_regular_bdry,cor:main_regular_bdry} will be presented in \cref{sec:proof_main}, following the strategy outlined in the introduction.

\subsection{Outlook}
\label{sec:outlook}

Here we explain a few natural applications or generalizations of our work.
\begin{itemize}
    \item \textbf{Spectral convergence rates:} The techniques of this paper, in particular the consistency between nonlocal and local energies proved in \cref{prop:consistency,prop:nonlocal2local}, as well as the the modification of the boundary values in \cref{sec:bdry_mod}, can be used in a straightforward way to prove convergence rates for fractional eigenvalues and eigenvectors, cf. \cite{BrPaSq16,FeSa20,Fogthesis}.
    In \cite{garcia2020error} similar techniques were used to prove spectral convergence rates of graph discretizations of the Laplace--Beltrami operator on a manifold. 
    It will be interesting to investigate whether the resulting fractional convergence rates are sharp or whether improved techniques have to be used, as done in \cite{calder2022improved} where the authors improved the results of \cite{garcia2020error} using a combination of variational and PDE techniques.
    \item \textbf{Nonlinear problems:} We expect that the techniques are also useful for fractional $p$-Laplacian problems since we hardly use the linearity of problems \labelcref{eq:fld,eq:ld}.
    At least in one spatial dimension, simple computations show that our techniques apply directly to the fractional $p$-Laplacian or to the convergence of the stochastic tug-of-war formulation of the game-theoretic $p$-Laplacian, proved in \cite{del2022convergence}. 
    In the general case, one would have to change the mollification procedure for constructing local competitors to operator on a length scale which is larger than the one we are using, see \cite{garcia2022graph} for a similar problem in the context of the $1$-Laplacian.    
    This reduces the convergence rate but still allows one to use the flexible variational approach from the present paper.
    \item \textbf{General operators, manifolds:} Further generalizations which we expect to be doable with a little more technical effort include fractional operators with more general kernels, anisotropic weights, or such defined on smooth manifolds instead of Euclidean domains.
\end{itemize}

\section{Variational setting: Consistency and stability results}

For proving the convergence rates, we will use the variational interpretations of the problems \labelcref{eq:fld,eq:ld}. 
For this let us define the fractional kernel $\eta_s:\R_+\to\R_+$, for $s\in(0,1)$, by
\begin{align}\label{eq:def_eta_C}
    \eta_s(t)= \frac{C_{s,d}}{t^{d+2s}} \quad \textup{with} \quad C_{s,d}=\frac{d}{\omega_d}(1-s),
\end{align}
where we recap that $\omega_d$ denotes the $(d-1)$-dimensional Hausdorff measure of the unit sphere in $\R^d$.
Then the $s$-fractional energy is defined as 
\begin{align}\label{eq:fractional_energy_Rd}
    J_s(u) := 
    \int_{\R^d}\int_{\R^d} \eta_s\left(\abs{x-y}\right)\abs{u(x)-u(y)}^2\d y\d x - \int_\Omega f_s(x) u(x) \d x,\quad 
    u \in W^{s,2}(\R^d),
\end{align}
and we consider the minimization problem
\begin{align}\label{eq:nonlocal_problem}
    \min\left\lbrace 
    J_s(u) \st 
    \,u=g\quad \text{on}\quad\R^d\setminus\Omega
    \right\rbrace,
\end{align}
which is equivalent to the weak formulation of the fractional Dirichlet problem \labelcref{eq:fld}.
Similarly, for the local problem we define the energy
\begin{align}\label{eq:local_energy_Rd}
    J(u) := \frac{1}{2} \int_{\R^d}\abs{\grad u(x)}^2\d x - \int_\Omega f(x) u(x) \d x,\qquad u \in W^{1,2}(\R^d).
\end{align}
The associated minimization problem is
\begin{align}\label{eq:local_problem}
    \min\left\lbrace 
    J(u) \st 
    u=g \quad \textup{on} \quad \R^d\setminus\Omega
    \right\rbrace
\end{align}
which is equivalent to weak formulation of the local Dirichlet problem \labelcref{eq:ld}.
For convenience we also define nonlocal and local Dirichlet energies as
\begin{alignat}{2}
    D_s(u) 
    &:= 
    \int_{\R^d}\int_{\R^d} \eta_s\left(\abs{x-y}\right)\abs{u(x)-u(y)}^2\d y\d x,\qquad &&u \in W^{s,2}(\R^d),\label{eq:Denergy-s}\\
    D(u) 
    &:=
    \frac{1}{2} \int_{\R^d}\abs{\grad u(x)}^2\d x ,\qquad &&u \in W^{1,2}(\R^d).\label{eq:Denergy}
\end{alignat}

\subsection{Variational consistency}

First, we show that for a function in $W^{1,2}(\R^d)$ the nonlocal energy can be controlled from above by the local one plus an error term of order $O(1-s)$. This result can be interpreted as a local to nonlocal (upper) consistency of the energies. Full consistency could be obtained for more regular functions, but this is not  necessary for the results of this paper.
\begin{proposition}\label{prop:consistency}
    Let $\phi\in W^{1,2}(\R^d)\cap L^p(\Omega)$ and $f,f_s\in L^q(\Omega)$ where $\frac{1}{p}+\frac{1}{q}=1$.
    Then it holds
    \begin{align*}
    J_s(\phi) \leq J(\phi) + \frac{2d}{s} \|\phi\|_{L^2(\R^d)}^2 (1-s)
    +\norm{\phi}_{L^p(\Omega)}\norm{f-f_s}_{L^q(\Omega)}.
    \end{align*}
\end{proposition}
\begin{proof}
We start by studying the $s$-fractional Dirichlet energy $D_s$ given by \labelcref{eq:Denergy-s}.
Let us write $D_s = D^1_s + D^2_s$ where
\begin{align}\label{eq:D1s}
    D^1_s(\phi)&:=\int_{\R^d}\int_{B_1(x)} \eta_s\left(\abs{x-y}\right)\abs{\phi(x)-\phi(y)}^2\d y\d x,
    \\\label{eq:D2s}
    D^2_s(\phi)&:=\int_{\R^d}\int_{\R^d\setminus B_1(x)} \eta_s\left(\abs{x-y}\right)\abs{\phi(x)-\phi(y)}^2\d y\d x.
\end{align}
First we note that
\begin{align*}
D^2_s(\phi)&= \int_{\R^d}\int_{\R^d\setminus B_1} \eta_s\left(\abs{z}\right)\abs{\phi(x)-\phi(x+z)}^2\d z\d x\\
&\leq2 \int_{\R^d\setminus B_1} \eta_s\left(\abs{z}\right) \int_{\R^d}\abs{\phi(x)}^2+\abs{\phi(x+z)}^2\d x\d z\\
&= 4 \|\phi\|_{L^2(\R^d)}^2 \int_{\R^d\setminus B_1} \eta_s\left(\abs{z}\right) \d z\\
&= 4 \|\phi\|_{L^2(\R^d)}^2 \omega_d C_{s,d} \int_1^\infty \frac{r^{d-1}}{r^{d+2s}}\d r\\
&= 4 \|\phi\|_{L^2(\R^d)}^2 \omega_d C_{s,d} \frac{1}{2s}\\
&= \frac{2d}{s}\|\phi\|_{L^2(\R^d)}^2 (1-s),
\end{align*}
where we used the definition of $C_{s,d}$ in \labelcref{eq:def_eta_C} in the last line.
On the other hand, by Jensen's inequality and a change of variables
\begin{align*}
D^1_s(\phi)&=\int_{\R^d}\int_{B_1} \eta_s\left(\abs{z}\right)\abs{\int_0^1 \nabla \phi(x+tz)\cdot z\d t}^2\d z\d x\\
&\leq \int_{\R^d}\int_{B_1} \eta_s\left(\abs{z}\right)\int_0^1 \abs{\nabla \phi(x+tz)\cdot z}^2\d t\d z\d x\\
& = \int_{\R^d}\int_{B_1} \eta_s\left(\abs{z}\right) \abs{\nabla \phi(x)\cdot z}^2\d z\d x.
\end{align*}
In the inner integral we make a change of variables of the form $z\to A^T z$ where $A\in\R^{d\times d}$ is an orthogonal matrix with $A\nabla\phi(x)=\abs{\nabla\phi(x)}e_1$.
This yields
\begin{align*}
    D^1_s(\phi)
    \leq 
    \int_{\R^d}\int_{B_1} \eta_s\left(\abs{z}\right) \abs{z_1}^2\abs{\nabla \phi(x)}^2\d z\d x
    = 
    \left(\int_{B_1} \eta_s\left(\abs{z}\right) \abs{z_1}^2\d z\right)
    \left(\int_{\R^d}\abs{\nabla \phi(x)}^2\d x \right).
\end{align*}
By the definition of $C_{s,d}$, it holds
\begin{align}\label{eq:cte}
\int_{B_1} \eta_s\left(\abs{z}\right) \abs{z_1}^2\d z= \frac{1}{d}\int_{B_1} \eta_s\left(\abs{z}\right) \abs{z}^2\d z = \frac{\omega_d}{d}C_{s,d} \int_0^1 r^{1-2s} \d r=
\frac{\omega_d C_{s,d}}{2d(1-s)}=\frac{1}{2},
\end{align}
leading to $D_s^1(\phi) \leq D(\phi)$.
To complete the proof we consider the linear terms and get
\begin{align*}
    -\int_\Omega f_s \phi \d x 
    &= 
    -\int_\Omega f \phi \d x 
    +\int_\Omega (f-f_s) \phi \d x 
    \\
    &\leq 
    -\int_\Omega f \phi \d x 
    +
    \norm{\phi}_{L^2(\R^d)}
    \norm{f-f_s}_{L^2(\R^d)}.
\end{align*}
Combining these estimates completes the proof.
\end{proof}

\subsection{Stability estimates of the energies}

Now we show that the distance between a minimizer of the energy $J$ or $J_s$, respectively, and an arbitrary function in the minimizing space is controlled by the difference of their energies. 
This can be seen as stability estimates of the energies associated to our problems and ultimately allows us to pull back the convergence rates we prove for the energies to the minimizers themselves.
The proofs are straightforward and simply use the definition of weak solutions.
Therefore, we postpone them to \cref{sec:stability}.

\begin{proposition}[Stability of the local problem]\label{prop:convexity_local}
Let $u\in W^{1,2}(\R^d)$ solve \labelcref{eq:local_problem}
and $C_{\Omega}>0$ be the Poincaré constant of $\Omega$. Then, for all $\phi \in W^{1,2}(\R^d)$ with $\phi=g$ in $\R^d\setminus\Omega$, we have
\begin{align*}
    \frac{1}{2}
    \norm{\nabla\phi-\nabla u}_{L^2(\Omega)}^2
    = J(\phi) - J(u).
\end{align*}
\end{proposition}
Similarly we also have a stability result for solutions of the nonlocal problem. 
\begin{proposition}[Stability of the nonlocal problem]\label{prop:convexity_nonlocal}
Let $u_s\in W^{s,2}(\R^d)$ solve \labelcref{eq:nonlocal_problem}.
For all $\phi \in W^{s,2}(\R^d)$ with $\phi=g$ in $\R^d\setminus\Omega$ it holds
\begin{align*}
    \frac12[\phi - u_s]_{W^{s,2}(\R^d)}^2 = J_s(\phi) - J_s(u_s).
\end{align*}
\end{proposition}

\section{Adapted smoothing operators}

In this section we will construct an adapted smoothing operator $I_s$ which maps a function in $W^{s,2}(\R^d)$ to one in $W^{1,2}(\R^d)$ in such a way that the Dirichlet energy of the smoothed function is controlled by the nonlocal Dirichlet energy of the function with constant one, i.e., $D(I_s[\phi])\leq D_s(\phi)$.
This last property excludes standard mollification procedures but requires a smoothing operation which is tailored to the fractional kernel in the definition of the nonlocal Dirichlet energy~$D_s$.
For constructing the mollification operator we first define some sort of antiderivative of the fractional kernel which will act as profile function for the smoothing operator.

\begin{definition}[Smoothing kernel]\label{def:kernel_psi_s_eps}
    For $\eps\in[0,1)$, $s\in(0,1)$, and $t\geq 0$ we define
    \begin{align*}
        \psi_s^\eps(t):=\frac{2}{1-\eps^{2-2s}}\left(\int_{\max\{\eps,t\}}^1\eta_s(\tau)\tau\d \tau\right)_+.
    \end{align*}
\end{definition}
Using the definition of the fractional kernel $\eta_s$ in \labelcref{eq:def_eta_C} it is obvious that the support of the kernel $\psi_s^\eps$ is in $[0,1]$ \emph{for all} values of $s\in(0,1)$.
At the same time we will show later that most of its mass is concentrated in the much smaller interval~$[0,1-s]$.

Note that the parameter $\eps>0$ is used to de-singularize the integral for values $t$ close to zero where the fractional kernel $\eta_s$ blows up. 
We will later see that for our purposes we can actually choose $\eps=0$.
First, we prove several properties of this function.
\begin{lemma}\label{lem:smoothlem}
We have the following properties for all $\eps\in[0,1)$:
\begin{enumerate}[\rm (a)]
    \item\label{smoothlem-support}
    $\displaystyle\supp\psi_s^\eps = [0,1]$.
    \item\label{smoothlem-mass} For all $\alpha\geq0$ we have that 
    \[\displaystyle\int_{\R^d} \psi_s^\eps(|z|)
    \abs{z}^\alpha
    \d z=
    \frac{d}{d+\alpha} \cdot\frac{(1-s)}{(1-s)+\frac{\alpha}{2}}\frac{1- \eps^{2-2s+\alpha}}{1-\eps^{2-2s}}.\]
    In particular, for $\alpha=0$, the integral equals one.
    \item\label{smoothlem-bound}  
    If $d\geq 2$, then $\displaystyle\psi_s^\eps(t) \leq \frac{1}{1-\eps^{2-2s}}\frac{2}{d+2s-2}t^2 \eta_{s}(t)$ for all $t\in(0,1)$.
    \item\label{smoothlem-bound_d=1}  
    If $d=1$, then $\displaystyle\psi_s^\eps(t) \leq \frac{2}{1-\eps^{2-2s}}t^{1+s} \eta_{s}(t)$ for all $t\in(0,1)$.
    \item\label{smoothlem-derivative}
    $\displaystyle(\psi^\eps_s)'(t)=-\frac{2}{1-\eps^{2-2s}}\eta_s(t)t \one_{(\eps,1)}(t)$ for all $t\in(0,1)\setminus\{\eps\}$.
\end{enumerate}
\end{lemma}
\begin{proof}
\cref{smoothlem-support} is obvious from the definition of $\psi_s^\eps$.
Let us prove \cref{smoothlem-mass}. 
Let us, for convenience, use the abbreviation $M(s,\eps):=\frac{2}{1-\eps^{2-2s}}$. 
Then,
\begin{align*}
  \int_{B_1} \psi_s^\eps(|z|)&|z|^\alpha\d z\\
  &= M(s,\eps) \omega_d \int_0^1 t^{d-1+\alpha} \int_{\max\{\eps,t\}}^1\eta_s(\tau)\tau\d \tau\d t \\
  &= M(s,\eps) \omega_d \int_0^\eps t^{d-1+\alpha} \int_{\eps}^1\eta_s(\tau)\tau\d \tau\d t +M(s,\eps) \omega_d \int_\eps^1 t^{d-1+\alpha} \int_{t}^1\eta_s(\tau)\tau\d \tau\d t \\
  &= M(s,\eps) \omega_d  \int_{\eps}^1\eta_s(\tau)\tau \int_0^\eps t^{d-1+\alpha}\d t \d \tau+M(s,\eps) \omega_d \int_\eps^1 \eta_s(\tau)\tau  \int_{\eps}^\tau t^{d-1+\alpha} \d t \d \tau\\
  &= M(s,\eps) \omega_d  \int_{\eps}^1\eta_s(\tau)\tau \int_0^\tau t^{d-1+\alpha}\d t \d \tau\\
  &= \frac{M(s,\eps) \omega_d}{d+\alpha} \int_{\eps}^1\eta_s(\tau)  \tau^{1+d+\alpha} \d \tau\\
  &=\frac{d}{d+\alpha} \cdot\frac{(1-s)}{(1-s)+\frac{\alpha}{2}} \cdot \frac{1- \eps^{2-2s+\alpha}}{1-\eps^{2-2s}}.
\end{align*}
To prove \cref{smoothlem-bound} we compute for $t\in(0,1)$:
\begin{align*}
    \psi_s^\eps(t)
    &
    \leq
    \frac{2}{1-\eps^{2-2s}}\int_{t}^1 \eta_s(\tau)\tau \d \tau  
    \\
    &= \frac{2}{1-\eps^{2-2s}}C_{s,d} \frac{t^{2-d-2s}-1 }{d+2s-2}
    \\
    &\leq \frac{2}{1-\eps^{2-2s}}\frac{C_{s,d}}{d+2s-2}t^{2-d-2s}.
\end{align*}
Similarly, one can prove \cref{smoothlem-bound_d=1} for $d=1$.
Starting at the second line of the previous chain of inequalities for $d=1$, we have, for $s\not=1/2$ 
\begin{align*}
    \psi_s^\eps(t)
    &\leq 
    \frac{2}{1-\eps^{2-2s}}C_{s,1} \frac{1- t^{1-2s}}{1-2s}
    \\
    &=
    \frac{2}{1-\eps^{2-2s}}
    t^{1+s}
    \eta_s(t)
    \frac{t^s-t^{1-s}}{1-2s}
    \\
    &\leq 
    \frac{2}{1-\eps^{2-2s}}
    t^{1+s}
    \eta_s(t),\qquad\forall t\in(0,1),\,s\in(0,1)\setminus\left\{\frac{1}{2}\right\}.
\end{align*}
If $s=1/2$, the estimate is
\begin{align*}
    \psi_s^\eps(t)
    &\leq 
    \frac{2}{1-\eps} C_{s,1} \log\left(\frac{1}{t}\right)
    =
    \frac{2}{1-\eps}
    t^{1+\frac{1}{2}}
    \eta_\frac{1}{2}(t)
    t^{\frac{1}{2}}\log\left(\frac{1}{t}\right)\\
    &\leq 
    \frac{2}{1-\eps^{2-2s}}
    t^{1+s}
    \eta_s(t),\quad\forall t\in(0,1).
\end{align*}
Finally, \cref{smoothlem-derivative} is a simple consequence of the definition of $\psi_s^\eps$.
\end{proof}

\subsection{Stability and consistency of the mollifier}

Now we will define a convolution operator which is adapted to the fractional problem and prove that it is close to the identity and furthermore is consistent with the Dirichlet energies as outlined above.
\begin{definition}
    Let $\phi\in L^2(\R^d)$. For $\eps\in[0,1)$, we define
    \[
    I_{s}^\eps[\phi](x):= \int_{\R^d} \psi_s^\eps(|x-y|) \phi(y) \d y
    =
    \int_{B_1(x)} \psi_s^\eps(|x-y|) \phi(y) \d y.
    \]
\end{definition}
Using Jensen's inequality it is standard to check that $I_s^\eps : L^p(\R^d)\to L^p(\R^d)$ with $\norm{I_s^\eps[\phi]}_{L^p(\R^d)}\leq\norm{\phi}_{L^p(\R^d)}$ for all $p\in[1,\infty]$.
In fact, restricted on the space $W^{s,2}(\R^d)$, the operator $I_s^\eps$ is an approximation of the identity.
More precisely, we prove that the $L^2(\R^d)$-error between $I_s^\eps[\phi]$ and $\phi$ is of order $\sqrt{1-s}$ for every $W^{s,2}$-function $\phi$.

\begin{proposition}\label{prop:convolution_id}
Let $\phi\in W^{s,2}(\R^d)$. Then for all $\eps\in[0,1)$ and all $s\in(0,1)$ it holds
\[
\|I_s^\eps[\phi]-\phi\|_{L^2(\R^d)}^2 \leq 
\begin{dcases}
    \left(\frac{2}{1-\eps^{2-2s}}\right)^2
    (1-s)
    D_s^1(\phi),
    \qquad&\text{if }d=1,
    \\
    \frac{2d}{(d+2s-2)^2(2-s)}\cdot \frac{1-s}{\left(1-\eps^{2-2s}\right)^2}
    D^1_s(\phi),\qquad&\text{if }d\geq 2,
\end{dcases}
\]
where $D^1_s$ is defined in \labelcref{eq:D1s}.
\end{proposition}
\begin{remark}
    If one is willing to assume Hölder regularity of $\phi$ and restricts the $L^2(\R^d)$-norm to a bounded set, one can get a quadratic dependence on $1-s$. 
    More precisely, if $\phi\in C^{0,s}(\R^d)$, then for any bounded open set $A\subset\R^d$ one can prove
    \begin{align*}
        \norm{I_s^\eps[\phi]-\phi}_{L^2(A)}^2
        \leq 
        \mathcal L^d(A)\left(\frac{2d}{(d+2s-2)(2-s)}[\phi]_{C^{0,s}(\R^d)}\right)^2\left(\frac{1-s}{1-\eps^{2-2s}}\right)^2.
    \end{align*}
    However, in the proof of our main results other errors of the same order as the one in \cref{prop:convolution_id} will be present and hence we will not use the improved estimate for Hölder functions.
\end{remark}
\begin{proof}[Proof of \cref{prop:convolution_id}]
We first consider the case $d\geq 2$.
Using \cref{smoothlem-mass,smoothlem-bound} in \cref{lem:smoothlem} to get for almost every $x\in \R^d$:
\begin{align*}
|I_s^\eps[\phi]&(x)-\phi(x)|
\\
&\leq \int_{B_1(x)} \psi_s^\eps(|x-y|) |\phi(x) -\phi(y)|\d y\\
&\leq \frac{2}{d+2s-2}\frac{1}{1-\eps^{2-2s}} \int_{B_1(x)} |x-y|^2\eta_{s}(|x-y|) |\phi(x) -\phi(y)|\d y\\
&\leq \frac{2}{d+2s-2}\frac{1}{1-\eps^{2-2s}}
\left(\int_{B_1} |z|^4\eta_{s}(|z|)\d z\right)^{\frac{1}{2}}\left(\int_{B_1(x)} \eta_{s}(|x-y|) |\phi(x) -\phi(y)|^2\d y\right)^{\frac{1}{2}}\\
&= \frac{2}{d+2s-2}\frac{1}{1-\eps^{2-2s}}\left(\frac{d}{4-2s}(1-s) \right)^{\frac{1}{2}}\left(\int_{B_1(x)} \eta_{s}(|x-y|) |\phi(x) -\phi(y)|^2\d y\right)^{\frac{1}{2}}.
\end{align*}
Thus, by squaring this estimate and integrating over $\R^d$ we obtain
\begin{align*}
\int_{\R^d}|I_s^\eps[\phi](x)-\phi(x)|^2 \d x \leq  \left(\frac{2}{d+2s-2}\right)^2 \frac{d}{4-2s}\frac{1-s}{\left(1-\eps^{2-2s}\right)^2} D_s^1(\phi)
\end{align*}
which proves the first claim.
In the case $d=1$, we use \cref{smoothlem-bound_d=1} in place of \cref{smoothlem-bound} which yields
\begin{align*}
    |I_s^\eps[\phi]&(x)-\phi(x)|
    \\
    &\leq 
    \frac{2}{1-\eps^{2-2s}}
    \int_{B_1(x)}\abs{x-y}^{1+s}
    \eta_s(\abs{x-y})\abs{\phi(x)-\phi(y)}\d y
    \\
    &\leq 
    \frac{2}{1-\eps^{2-2s}}
    \left(
    \int_{B_1}\abs{z}^{2+2s}\eta_s(\abs{z})\d z
    \right)^\frac{1}{2}
    \left(
    \int_{B_1(x)}
    \eta_s(\abs{x-y})\abs{\phi(x)-\phi(y)}^2\d y
    \right)^\frac{1}{2}.
\end{align*}
The first bracket can be computed as follows:
\begin{align*}
    \int_{B_1}\abs{z}^{2+2s}\eta_s(\abs{z})\d z
    =
    (1-s)
    \int_0^1 
    t\d t
    \leq
    (1-s)
\end{align*}
which completes the proof.
\end{proof}
The following proposition proves that for Hölder functions, analogous to \cref{prop:convolution_id}, the convolution $I_s^\eps$ is close to the identity in uniform way.
\begin{proposition}\label{prop:convolution_pointwise}
    If $\phi\in C^{0,s}(\R^d)$, then it holds
    \begin{align*}
        \|I_s^\eps[\phi]-\phi\|_{L^\infty(\R^d)} \leq \frac{2d}{(d+s)(2-s)}[\phi]_{C^{0,s}(\R^d)}\frac{1-s}{1-\eps^{2-2s}}
        \leq 
        2[\phi]_{C^{0,s}(\R^d)}\frac{1-s}{1-\eps^{2-2s}}.
    \end{align*}
\end{proposition}
\begin{proof}
    We compute 
    \begin{align*}
        \abs{I_s^\eps[\phi](x)-\phi(x)} 
        \leq 
        \int_{\R^d}
        \psi_s^\eps(\abs{x-y})\abs{\phi(x)-\phi(y)}\d y
        \leq 
        [\phi]_{C^{0,s}(\R^d)}
        \int_{B_1}
        \psi_s^\eps(\abs{z})\abs{z}^s\d z
    \end{align*}
    and the statement follows
    as a consequence of \cref{smoothlem-mass} in \cref{lem:smoothlem} with $\alpha:=s$. Since the final estimate does not depend on $x\in\R^d$, we get the result for the $L^\infty(\R^d)$ norm.
\end{proof}

Now we prove the most important property of the smoothing operator $I_s^\eps$, namely that it is consistent with our energies.
That means, the Dirichlet energy of $I_s^\eps[\phi]$ can be controlled by the fractional Dirichlet energy of $\phi$ times a constant that is close to one (and equals one for $\eps=0$).
We also prove that the Lipschitz constant of $I_s^\eps[\phi]$ is bounded by the Hölder constant of $\phi$. 
Interestingly, both bounds are uniform in $\eps$ which allows us to choose $\eps=0$ in the sequel.
\begin{proposition}\label{prop:nonlocal2local}
Let $\phi\in W^{s,2}(\R^d)$. Then, for all $\eps\in(0,1)$, we have that $I_s^\eps[\phi]\in W^{1,2}(\R^d)$ and
\begin{align}\label{eq:nonlocal2local}
    D(I^\eps_s[\phi]) \leq \frac{1}{\left(1-\eps^{2-2s}\right)^2}
    D_s^1(\phi),
\end{align}
where $D^1_s$ is defined in \labelcref{eq:D1s} and $D$ is defined in \labelcref{eq:Denergy}.
Moreover, if $\phi\in C^{0,s}(\R^d)$, then $I^\eps_s[\phi]\in \Lip(\R^d)$ and
\begin{align}\label{eq:Lipmol}
    \|\nabla I_s^\eps[\phi]\|_{L^\infty(\R^d)}
    \leq 
    \frac{2d[\phi]_{C^{0,s}(\R^d)}}{1-\eps^{2-2s}}(1-\eps^{1-s})
    \leq
    \frac{2d[\phi]_{C^{0,s}(\R^d)}}{1-\eps^{2-2s}}.
\end{align}
\end{proposition}
\begin{remark}
    Let us remark that, for instance, the choice $\eps=\exp\left(\frac{\log(1-s)}{1-s}\right)=(1-s)^\frac{1}{1-s}$ satisfies $\eps\to 0$ exponentially fast as $s\to 1$ and furthermore
    \begin{align*}
        \frac{1}{1-\eps^{2-2s}}
        =
        \frac{1}{1-(1-s)^2}
        =
        1 + 
        \frac{(1-s)^2}{s(2-s)}.
    \end{align*}
    Thus, the corresponding factors on the right hand side of \labelcref{eq:nonlocal2local,eq:Lipmol} are quadratically close to one if $s\to 1$. However, we will even be able to choose  $\eps=0$ in some cases.
\end{remark}
\begin{proof}[Proof of \cref{prop:nonlocal2local}]
We compute the gradient of $I^\eps_s[\phi]$  using \cref{smoothlem-derivative} in \cref{lem:smoothlem}:
\begin{align*}
    \nabla I_s^\eps[\phi](x) &= \nabla \int_{\R^d} \psi_s^\eps(|x-y|) \phi(y)\d y \\
    &=\int_{\R^d} (\psi_s^\eps)'(|x-y|) \frac{x-y}{|x-y|}\phi(y)\d y\\
    &=- \frac{2}{1-\eps^{2-2s}} \int_{B_1(x)\setminus B_\eps(x)} \eta_s(|x-y|) (x-y)\phi(y)\d y \\
    &= \frac{2}{1-\eps^{2-2s}} \int_{B_1(x)\setminus B_\eps(x)} \eta_s(|x-y|) (x-y)(\phi(x)-\phi(y))\d y.
\end{align*}
In the computation above we have strongly used that $\eta_s(|z|)$ is bounded away from the origin and even (so that $\eta_s(|z|)z$ is odd, and thus integrates to zero on $B_1\setminus B_\eps$). 

Without loss of generality we can assume that $\nabla I_s^\eps[\phi](x) \neq 0$.
Also, using the Hölder inequality it is easy to see that $\abs{\nabla I_s^\eps[\phi](x)}<\infty$ if $\eps>0$.
Since we need a slightly more refined argument for upper-bounding the modulus $\abs{\nabla I_s^\eps[\phi](x)}$, let us consider $\xi(x)\in \partial B_1$ such that 
\[
 |\nabla I_s^\eps[\phi](x)|= \nabla I_s^\eps[\phi](x)\cdot \xi(x)
\]
and as a consequence we have
\begin{align}
    \label{eq:almostlipmol}
    |\nabla I_s^\eps[\phi](x)| 
    =
    \frac{2}{1-\eps^{2-2s}}
    \int_{B_1(x)\setminus B_\eps(x)} \eta_s(|x-y|) (x-y)\cdot \xi(x)(\phi(x)-\phi(y))\d y.
\end{align}
Using the Hölder inequality we obtain
\begin{align*}
    |\nabla &I_s^\eps[\phi](x)| \\
    &\leq 
    \frac{2}{1-\eps^{2-2s}}
    \left(\int_{B_1} \eta_s(|z|) (z\cdot \xi(x))^2\d z\right)^{\frac{1}{2}} \left(\int_{B_1(x)\setminus B_\eps(x)} \eta_s(|x-y|)\abs{\phi(x)-\phi(y)}^2\d y\right)^{\frac{1}{2}}\\
    &\leq
    \frac{2}{1-\eps^{2-2s}}
    \left(\int_{B_1} \eta_s(|z|) \abs{z_1}^2\d z\right)^{\frac{1}{2}} \left(\int_{B_1(x)\setminus B_\eps(x)} \eta_s(|x-y|)\abs{\phi(x)-\phi(y)}^2\d y\right)^{\frac{1}{2}}
    \\
    &=\frac{\sqrt{2}}{1-\eps^{2-2s}}\left(\int_{B_1(x)\setminus B_\eps(x)} \eta_s(|x-y|)\abs{\phi(x)-\phi(y)}^2\d y\right)^{\frac{1}{2}}.
\end{align*}
where the last equality follows from \labelcref{eq:cte}.
As a consequence, we conclude that
\begin{align*}
D(I_s^\eps[\phi])
&= 
\frac{1}{2}\int_{\R^d}|\nabla I_s^\eps[\phi](x)|^2\d x 
\leq
\frac{1}{\left(1-\eps^{2-2s}\right)^2}
\int_{\R^d} \int_{B_1(x)} \eta_s(|x-y|)\abs{\phi(x)-\phi(y)}^2\d y \d x
\\
&=
\frac{1}{\left(1-\eps^{2-2s}\right)^2}
D_s^1(\phi).
\end{align*}
To prove \labelcref{eq:Lipmol}, we come back to \labelcref{eq:almostlipmol} and use the assumption $\phi\in C^{0,s}(\R^d)$ to get
\begin{align*}
    |\nabla I_s^\eps[\phi](x)|     
    &=
    \frac{2}{1-\eps^{2-2s}}
    \int_{B_1(x)\setminus B_\eps(x)} \eta_s(|x-y|) (x-y)\cdot \xi(x)(\phi(x)-\phi(y))\d y
    \\
    &\leq 
    \frac{2[\phi]_{C^{0,s}(\R^d)}}{1-\eps^{2-2s}}
    d(1-s) \int_\eps^1 t^{d-1}t^{-d-2s} t^{1+s}\d t
    \\
    &=
    \frac{2[\phi]_{C^{0,s}(\R^d)}}{1-\eps^{2-2s}}
    d(1-\eps^{1-s}).
\end{align*}
Since this estimate does not depend on $x\in\R^d$, we get the same bound for $\norm{\nabla I_s^\eps[\phi]}_{L^\infty(\R^d)}$.
\end{proof}

\begin{remark}
    Defining the operator $I_s := I_s^0$, using that \labelcref{eq:nonlocal2local,eq:Lipmol}are uniform in $\eps$, and utilizing lower semicontinuity, the previous proposition implies that in fact
    \begin{align*}
        D(I_s[\phi]) \leq 
        D_s^1(\phi)
    \end{align*}
    and, if $\phi\in C^{0,s}(\R^d)$, then also
    \begin{align*}
        \|\nabla I_s[\phi]\|_{L^\infty(\R^d)}
        \leq 
        2d[\phi]_{C^{0,s}(\R^d)}.
    \end{align*}
    However, proving this requires some technicalities, in particular, to make sense of the gradient $\nabla I_s[\phi]$ as a principal value integral.
    Instead, we will keep $\eps$ positive for now and only send it to zero in the very end, in the proofs of \cref{thm:main,thm:less_regular_bdry}.
\end{remark}

\subsection{Tail bounds of the mollifier}

Finally we need tail bounds for gradient of the convolution operator $I_s^\eps$.
Although its kernel $\psi_s^\eps$ defined in \cref{def:kernel_psi_s_eps} has support in $[0,1]$, most of its mass is concentrated in a much smaller ball of radius $\rho$ which goes to zero as $s\to 1$.
Later we will choose $\rho=1-s$, potentially up to a logarithmic factor.
\begin{proposition}[Tail bounds]\label{prop:tail_bounds}
    Let $\phi\in C^{0,\alpha}(\R^d)$ for some $\alpha\in(0,1]$. Then, for all $s\in\left(0,1\right)$,  all $0<\eps<\rho\leq 1$ and all $x\in \R^d$, it holds
    \begin{align}
        \label{eq:grad_tail_conv_holder}
        \abs{\nabla I_s^\eps[\phi](x)-\int_{B_\rho(x)}\nabla_x\psi_s^\eps(\abs{x-y})\phi(y)\d y}
        &\leq 
        \frac{2d[\phi]_{C^{0,\alpha}(\R^d)}(1-s)}{\alpha+1-2s}
        \frac{1-\rho^{\alpha+1-2s}}{1-\eps^{2-2s}}.
    \end{align}
    In particular, for $\alpha=s$ we get
    \begin{align}
        \label{eq:grad_tail_conv_holder_s}
        \abs{\nabla I_s^\eps[\phi](x)-\int_{B_\rho(x)}\nabla_x \psi_s^\eps(\abs{x-y})\phi(y)\d y}
        &
        \leq 
        2d[\phi]_{C^{0,s}(\R^d)}
        \frac{1-\rho^{1-s}}{1-\eps^{2-2s}}.
    \end{align}
\end{proposition}
\begin{remark}
    We remark that the right hand side in \labelcref{eq:grad_tail_conv_holder} is always non-negative since the two terms $\alpha+1-2s$ and $1-\rho^{\alpha+1-2s}$ change signs from negative to positive at $\alpha = 2s-1$.
\end{remark}
\begin{proof}
As in the proof of \cref{prop:nonlocal2local}, it holds
\begin{align*}
    \nabla I_s^\eps[\phi](x)
    =
    M(s,\eps)
    \int_{B_1(x)\setminus B_\eps(x)}
    \eta_s(\abs{x-y})(x-y)(\phi(x)-\phi(y))\d y.
\end{align*}
As a consequence of \cref{smoothlem-derivative} in  \cref{lem:smoothlem}, we obtain
\begin{align*}
    \bigg|\nabla I_s^\eps[\phi](x)
    -
    \int_{B_\rho(x)}&\nabla_x\psi_s^\eps(\abs{x-y})\phi(y)\d y\bigg|
    \\
    &=\bigg|
    \nabla I_s^\eps[\phi](x)
    -
    M(s,\eps)\int_{B_\rho(x)\setminus B_\eps(x)}\eta_s(\abs{x-y})(x-y)(\phi(x)-\phi(y))\d y\bigg|
    \\
    &=
    \bigg|M(s,\eps)
    \int_{B_1(x)\setminus B_\rho(x)}
    \eta_s(\abs{x-y})(x-y)(\phi(x)-\phi(y))\d y\bigg|\\
    &\leq
    M(s,\eps)
    [\phi]_{C^{0,\alpha}(\R^d)}
    \int_{B_1\setminus B_\rho}
    \eta_s(\abs{z})
    \abs{z}^{1+\alpha}\d z
    \\
    &=
    \frac{2d
    [\phi]_{C^{0,\alpha}(\R^d)}
    (1-s)}{\alpha+1-2s}
    \frac{1-\rho^{\alpha+1-2s}}{1-\eps^{2-2s}}.
\end{align*}
\end{proof}

\section{Nonlocal to local convergence rates}

This section is devoted to proving the main results of the paper: \cref{thm:main,thm:less_regular_bdry}.

The proof will divided into three steps: 
First, we use the nonlocal solution $u_s$ to construct a competitor (a feasible function) for the local minimization problem \labelcref{eq:local_problem}, then we note that the local solution $u$ is automatically a competitor  for the nonlocal problem \labelcref{eq:nonlocal_problem}, and finally we use the stability of the local and the nonlocal problem from \cref{prop:convexity_local,prop:convexity_nonlocal} to prove the final convergence rate.
The most challenging part is the construction of the competitor for the local problem since this requires smoothing the nonlocal solution with the operator $I_s^\eps$ and adapting its boundary values.

\subsection{Boundary modifications}
\label{sec:bdry_mod}

Let $u_s\in W^{s,2}(\R^d)$ denote the unique solution of \labelcref{eq:nonlocal_problem}.
We abbreviate the set of points in $\Omega$ with distance at most $r>0$ to the boundary by 
\begin{align*}
    \partial_r\Omega := \{x\in\Omega\st \dist(x,\Omega^c)\leq r\}.
\end{align*}
\begin{remark}\label{rem:rdomainmeas}
   Note that if $\Omega$ has a Lipschitz boundary, then the measure of $\partial_r\Omega$ is proportional to $r$ and, in particular, converges to zero at the rate $O(r)$ as $r\to0^+$. That is, there exists a constant $C_\Omega\geq0$ such that $|\partial_r \Omega|\leq C_{\Omega} r$,
\end{remark}

Ideally, we would like to use the mollified version of $u_s$ given by $I_s^\eps[u_s]$ as a feasible function for the nonlocal problem \labelcref{eq:nonlocal_problem}. However, the effect of mollification makes $I_s^\eps[u_s]$ not being equal to $g$ outside of $\Omega$. We shall construct a feasible function for the nonlocal problem \labelcref{eq:nonlocal_problem} by linearly interpolating the complementary boundary data $g$ and the mollification $I_s^\eps[u_s]$ in $\partial_r\Omega$, by defining
\begin{equation}\label{eq:boundmol}
    w_s^{r,\eps}(x)=
    \begin{dcases}
        g(x), \quad&x\in \R^d\setminus \Omega\\
        \left(1-\frac{\dist(x,\Omega^c)}{r}\right)g(x) + \frac{\dist(x,\Omega^c)}{r} I_s^\eps[u_s](x), \quad &x\in\partial_r\Omega,
        \\
        I_s^\eps[u_s](x), \quad &x\in\Omega\setminus\partial_r\Omega.
    \end{dcases}
\end{equation}

We will prove now several properties of $w_s^{r,\eps}$. First, we will show that $w_s^{r,\eps}$ and $I_s^\eps[u_s]$ are close at points $x\in \partial_r\Omega$ and, consequently, also in $L^2(\Omega)$. 
For that purpose, we will assume that $u_s \in C^{0,s}(\R^d)$ which is guaranteed by \cite[Proposition 1.1]{ros-oton2014dirichlet} in the case $g\equiv 0$ and which remains a hypothesis for $g\in C^{1,\alpha}(\R^d)$, cf. \cref{rem:hölder,rem:hoelder?}. 

\begin{lemma}\label{lem:suit}
Let $\Omega$ be a Lipschitz domain, $u_s,g\in C^{0,s}(\R^d)$ and $w_{s}^{r,s}$ be defined by \labelcref{eq:boundmol}. Then, for all $x\in\partial_r\Omega$, it holds
\begin{align*}
    \abs{I_s^\eps[u_s](x)-w_s^{r,\eps}(x)}
    &\leq
    \abs{I_s^\eps[u_s](x)-g(x)}\\
    &\leq 2
    \left([u_s]_{C^{0,s}(\R^d)} + [g]_{C^{0,s}(\R^d)}\right)\left(r^s+
    \frac{1-s}{1-\eps^{2-2s}}\right).
\end{align*}
\end{lemma}
\begin{proof}
    
    Note first that since $u_s\in C^{0,s}(\R^d)$, then, given any two points $x,\overline{x}\in \R^d$, we have that
    \begin{align}
        |I_s^\eps[u_s](x)-I_s^\eps[u_s](\overline{x})| &\leq \int_{\R^d} \psi_s^\eps (|z|) \abs{u_s(x+z)-u_s(\overline{x}+z)}\d z \nonumber\\
        &\leq  [u_s]_{C^{0,s}(\R^d)} |x-\overline{x}|^s \int_{\R^d} \psi_s^\eps (|z|) \d z\nonumber\\
        &= [u_s]_{C^{0,s}(\R^d)} |x-\overline{x}|^s. \label{eq:HolderI}
    \end{align}
  This shows that $I_s^\eps[u]\in C^{0,s}(\R^d)$ and has the same modulus of continuity as $u_s$. 
    Let now, for $x\in \Omega$, denote $\pi(x)$ a closest point to $x$ in $\partial\Omega$, so that $\dist(x,\Omega^c)=|x-\pi(x)|$. 
    Then for $x\in\partial_r\Omega$ we get
    \begin{align*}
        |w_s^{r,\eps}(x)-&I_s^\eps[u_s](x)|
        \\
        &=
        \abs{1-\frac{\dist(x,\Omega^c)}{r}}
        \abs{I_s^\eps[u_s](x)-g(x)}
        \\
        &\leq 
        \abs{I_s^\eps[u_s](x)-I_s^\eps[u_s](\pi(x))}
        +
        \abs{g(x)-g(\pi(x))}
        +
        \abs{u_s(\pi(x))-I_s^\eps[u_s](\pi(x))},
    \end{align*}
    where we have used the definition of $w_s^{r,\eps}$ given by \labelcref{eq:boundmol} and the fact that,  since $\pi(x)\in \partial\Omega$, it holds $u_s(\pi(x))=g(\pi(x))$.
    We estimate the three terms using \labelcref{eq:HolderI}, $g\in C^{0,s}(\R^d)$, and \cref{prop:convolution_pointwise}, respectively, and obtain
    \begin{align*}
        \abs{w_s^{r,\eps}(x)-I_s^\eps[u_s](x)}
        &\leq 
        \left([u_s]_{C^{0,s}(\R^d)} + [g]_{C^{0,s}(\R^d)}\right)r^s
        +
        2[u_s]_{C^{0,s}(\R^d)}
        \frac{1-s}{1-\eps^{2-2s}}
    \end{align*}  
    from which the result trivially follows.
\end{proof}
We get the following simple corollary from \cref{lem:suit}.
\begin{corollary}\label{cor:L2_error_modification}
Let the assumptions of \cref{lem:suit} hold. Then the exists a constant $C_\Omega\geq0$ such that
    \begin{align*}
        \|I_s^\eps[u_s]-w_s^{r,\eps}\|_{L^2(\Omega)}^2
        \leq& 
        C_\Omega\left([u_s]_{C^{0,s}(\R^d)}^2 + [g]_{C^{0,s}(\R^d)}^2\right)  \left( r^{1+2s} +\left(\frac{1-s}{1-\eps^{2-2s}}\right)^2 r\right).
    \end{align*}   
\end{corollary}
\begin{proof}
Note that, by the definition of $w_s^{r,\eps}$ given by \labelcref{eq:boundmol}, we have that
\begin{align*}
     \|I_s^\eps[u_s]-w_s^{r,\eps}\|_{L^2(\Omega)}^2&= \int_{\Omega} |I_s^\eps[u_s](x)-w_s^{r,\eps}(x)|^2 \d x\\
     &= \int_{\partial_r\Omega} |I_s^\eps[u_s](x)-w_s^{r,\eps}(x)|^2 \d x\\
     & \leq  \left(2
    \left([u_s]_{C^{0,s}(\R^d)} + [g]_{C^{0,s}(\R^d)}\right)\left(r^s+
    \frac{1-s}{1-\eps^{2-2s}}\right)\right)^2 |\partial_r \Omega|,
\end{align*}
where we have used \cref{lem:suit,rem:rdomainmeas}.
\end{proof}

We will show now the local energies of $w_s^{r,\eps}$ and $I_s^\eps[u_s]$ are also close. We do it both for $g\in C^{1,\alpha}(\R^d)$ and $g\equiv 0$.

\begin{lemma}\label{lem:lastlemma}
Let $\Omega$ be a Lipschitz domain, $u_s\in C^{0,s}(\R^d)$ be the solution of \labelcref{eq:nonlocal_problem}, and $w_{s}^{r,\eps}$ be defined by \labelcref{eq:boundmol}. The following assertions hold:
\begin{enumerate}[\rm (a)]
\item\label{lem-suitgrad-item2} If $g\in C^{1,\alpha}(\R^d)\cap W^{s,2}(\R^d)$ and $\nabla g\in \dot W^{\beta,1}(\R^d)$ for some $\alpha,\beta\in(0,1]$, then for all $\rho\in(\eps,1)$ we have that
\begin{align*}
    |J(w_s^{r,\eps}&)-J(I_s^\eps[u_s])| 
    \leq 
    C\left([u_s]_{C^{0,s}(\R^d)}^2 + \|g\|_{C^{1,\alpha}(\R^d)}^2+
    \norm{g}_{C^{0,1}(\R^d)}[g]_{\dot W^{\beta,1}(\R^d)}\right) \times
    \\
    &
      \times \Bigg[ r^{2s-1}  + 
    \frac{r+\rho+ (1-\rho^{2-2s}) + \frac{1-s}{\alpha}\left(\rho^{2+\alpha-2s}-\eps^{2+\alpha-2s}\right) + (1-s)^2r^{-1}}{(1-\eps^{2-2s})^2}
    \Bigg]
    \\
    &\qquad \qquad \qquad 
    +
    C
    \left([u_s]_{C^{0,s}(\R^d)} + [g]_{C^{0,1}(\R^d)}\right)
    \norm{f}_{L^\infty(\Omega)}
    \left(
    r^{1+s}
    +
    \frac{1-s}{1-\eps^{2-2s}}r
    \right)
\end{align*}
where the constant $C>0$ depends solely on $\Omega$ and $d$.
\item\label{lem-suitgrad-item3} If $u_s\in C^{0,s}(\R^d)$ and $g\equiv 0$ on $\R^d\setminus\Omega$, then, for all $\rho\in(\eps,1)$, we have that
\begin{align*}
    |J(w_s^{r,\eps})-J(I_s^\eps[u_s])| 
    &\leq 
    C
    [u_s]_{C^{0,s}(\R^d)}^2
    \Bigg[r^{2s-1} 
    + 
    \frac{ r+\rho+ (1-\rho^{1-s})^2  + (1-s)^2r^{-1}}{(1-\eps^{2-2s})^2}
    \Bigg]
    \\
    &\qquad
    +
    C
    [u_s]_{C^{0,s}(\R^d)}
    \norm{f}_{L^\infty(\Omega)}
    \left(
    r^{1+s}
    +
    \frac{1-s}{1-\eps^{2-2s}}r
    \right)
\end{align*}
where the constant $C>0$ depends solely on $\Omega$ and $d$.
\end{enumerate}
\end{lemma}
\begin{proof}
In this proof we shall use the notation $a\lesssim b$ for $a,b\in\R$ which means that $a\leq C b$ for a constant $C$ that just depends on the domain $\Omega$ and on the dimension $d$. 

Let us prove \cref{lem-suitgrad-item2} first. For brevity we write $d(x)$ instead of $\dist(x,\Omega^c)$. 
Since both $g$, and $I_s^\eps[u_s]$, and the distance function $d$ are Lipschitz continuous, the same holds for the function $w_s^{r,\eps}$ and we get for almost every $x\in\Omega$:
\begin{equation*}
\nabla w_s^{r,\eps}(x)\hspace{-0.5mm}=\hspace{-1mm}
    \begin{dcases}
    \nabla g(x) &\hspace{-1.65mm} x\in \R^d\hspace{-0.5mm}\setminus \hspace{-0.5mm}\Omega,\\
    \left(\hspace{-0.5mm} 1-\frac{d(x)}{r}\hspace{-0.5mm}\right)\nabla g(x) + \frac{d(x)}{r} \nabla I_s^\eps[u_s](x) + \frac{\nabla d(x)}{r} \left(I_s^\eps[u_s](x)-g(x) \right)  &\hspace{-1.65mm} x\in \partial_r\Omega,
    \\
    \nabla I_s^\eps[u_s](x)  & \hspace{-1.65mm}x\in\Omega\hspace{-0.5mm}\setminus\hspace{-0.5mm}\partial_r\Omega.
    \end{dcases}
\end{equation*}
We start by estimating the difference of the Dirichlet energies of $w_s^{r,\eps}$ and $I_s^\eps[u_s]$, i.e.,
\[
|D(w_s^{r,\eps})-D(I_s^\eps[u_s])| = \left|\int_{\R^d}|\nabla w_s^{r,\eps}(x)|^2\d x -  \int_{\R^d}|\nabla I_s^\eps[u_s](x)|^2\d x \right|.
\]
First note that, for $x\in \Omega\setminus\partial_r\Omega$, we have that $\nabla w_s^{r,\eps}(x)=\nabla I_s^\eps[u_s](x) $, and thus,
\begin{align*}
|D(w_s^{r,\eps})-D(I_s^\eps[u_s])| \leq& \left|\int_{\R^d\setminus \Omega}|\nabla w_s^{r,\eps}(x)|^2\d x -  \int_{\R^d\setminus \Omega}|\nabla I_s^\eps[u_s](x)|^2\d x \right|\\
&+ \left|\int_{\partial_r \Omega}|\nabla w_s^{r,\eps}(x)|^2\d x -  \int_{\partial_r \Omega}|\nabla I_s^\eps[u_s](x)|^2\d x \right|
=G_1+G_2.
\end{align*}
Let us first estimate $G_2$. 
Using that $\abs{\nabla d}=1$ almost everywhere, we have for almost every $x\in \partial_r \Omega$ that
\[
|\nabla w_s^{r,\eps}(x)|\leq  |\nabla g(x)| +  |\nabla I_s^\eps[u_s](x)| + \frac{1}{r} |I_s^\eps[u_s](x)-g(x) |
\]
and thus
\[
|\nabla w_s^{r,\eps}(x)|^2\lesssim |\nabla g(x)|^2 +  |\nabla I_s^\eps[u_s](x)|^2 + \frac{1}{r^2} |I_s^\eps[u_s](x)-g(x) |^2.
\]
This implies that,
\begin{align*}
G_2 &\lesssim \int_{\partial_r \Omega} |\nabla g(x)|^2 \d x + \int_{\partial_r \Omega} |\nabla I_s^\eps[u_s](x)|^2 \d x + \frac{1}{r^2}\int_{\partial_r \Omega} |I_s^\eps[u_s](x)-g(x) |^2 \d x\\
&= G_2^1+G_2^2+G_2^3.
\end{align*}
We trivially bound $G_2^1 \lesssim [g]_{C^{0,1}(\R^d)}^2 r$. For $G_2^2$, we use the estimate \labelcref{eq:Lipmol} to get
\begin{align*}
G_2^2 \lesssim \|\nabla I_s^\eps[u_s]\|_{L^\infty(\R^d)}^2 r
\leq \left(\frac{2d[u_s]_{C^{0,s}(\R^d)}}{1-\eps^{2-2s}}\right)^2 r
\lesssim [u_s]_{C^{0,s}(\R^d)}^2 \frac{r}{(1-\eps^{2-2s})^2}.    
\end{align*}
To estimate $G_2^3$ we use  \cref{lem:suit} to get
\begin{align*}
    G_2^3 \lesssim \left([u_s]_{C^{0,s}(\R^d)}^2 + [g]_{C^{0,s}(\R^d)}^2\right)\left(r^{2s-1}+
     \left(\frac{1-s}{1-\eps^{2-2s}}\right)^2\frac{1}{r}\right).
\end{align*} 
With this, the estimate for $G_2$ is finished and we obtain
\begin{align}\label{eq:estim_G2}
    G_2 \lesssim  \left([u_s]_{C^{0,s}(\R^d)}^2 + [g]_{C^{0,1}(\R^d)}^2\right)
    \left[\frac{r}{(1-\eps^{2-2s})^2} + r^{2s-1} + \frac{1}{(1-\eps^{2-2s})^2}\frac{(1-s)^2}{r}\right].
\end{align}
Let us now estimate $G_1$. First, we note that, if $x\in \R^d\setminus \Omega$, then $w_s^{r,\eps}(x)=g(x)$ and thus
\[
G_1= \left|\int_{\R^d\setminus \Omega}|\nabla g(x)|^2\d x -  \int_{\R^d\setminus \Omega}|\nabla I_s^\eps[u_s](x)|^2\d x \right|
\]
In analogy to $\partial_r\Omega$, we denote by $\partial^\rho \Omega := \{x\in \R^d\setminus \Omega \st \dist(x,\Omega)<\rho \}$ the outer strip around~$\Omega$ with width $\rho>0$. 
Then,
\begin{align*}
  G_1 
  &\leq  \left|\int_{\partial^\rho \Omega}|\nabla g(x)|^2\d x\right| +\left|  \int_{\partial^\rho \Omega}|\nabla I_s^\eps[u_s](x)|^2\d x \right|
  \\
  &\qquad + \left|\int_{\partial^2\Omega\setminus \partial^\rho \Omega}|\nabla g(x)|^2\d x -  \int_{\partial^2\Omega\setminus \partial^\rho \Omega}|\nabla I_s^\eps[u_s](x)|^2\d x \right|
  \\
  &\qquad
  +
  \left|\int_{\R^d\setminus(\partial^2\Omega\cup\Omega)}|\nabla g(x)|^2\d x -  \int_{\R^d\setminus(\partial^2\Omega\cup\Omega)}|\nabla I_s^\eps[g](x)|^2\d x \right|
  \\
  &=: G_1^1+G_1^2+G_1^3+G_1^4,
\end{align*}
We proceed as in the estimates for $G_2^1$ and $G_2^2$, to get
\[
G_1^1 \lesssim [g]_{C^{0,1}(\R^d)}^2\rho, \quad \textup{and} \quad G_1^2 \lesssim [u_s]_{C^{0,s}(\R^d)}^2 \frac{\rho}{(1-\eps^{2-2s})^2}.
\]
Next, we need to estimate $G_1^3$. 
For this purpose, we will make use of the tail estimates in \cref{prop:tail_bounds}, and in particular \cref{eq:grad_tail_conv_holder_s}.

Using the identity $\abs{a}^2-\abs{b}^2=(a+b)\cdot(a-b)\leq\left(\abs{a}+\abs{b}\right)\abs{a-b}$ for $a,b\in\R^d$ as well as \labelcref{eq:Lipmol} we get
\begin{align*}
    G_1^3
    &\leq 
    \left([g]_{C^{0,1}(\R^d)} + \| \nabla I_s^\eps[u_s]\|_{L^\infty(\R^d)}\right) \int_{\partial^2\Omega\setminus \partial^\rho \Omega}|\nabla g(x)-\nabla I_s^\eps[u_s](x)|\d x
    \\
    &\leq 
    \left([g]_{C^{0,1}(\R^d)} +\frac{2d[u_s]_{C^{0,s}(\R^d)}}{1-\eps^{2-2s}}\right) \underbrace{\int_{\partial^2\Omega\setminus \partial^\rho \Omega}|\nabla g(x)-\nabla I_s^\eps[u_s](x)|\d x}_{=:I}
\end{align*}
and so it suffices to estimate the integral $I$ on the right hand side using the tail bounds.
By \labelcref{eq:grad_tail_conv_holder_s}, we have
\[
\abs{\nabla I_s^\eps[u_s](x)-\int_{B_\rho(x)}\nabla_x \psi_s^\eps(\abs{x-y})u_s(y)\d y}
        \leq 
        2d[u_s]_{C^{0,s}(\R^d)}
        \frac{1-\rho^{1-s}}{1-\eps^{2-2s}}.
\]
Using this, together with $u_s=g$ in $\Omega^c$, we obtain with the triangle inequality 
\begin{align*}
    I
    &\leq  
    \int_{\partial^2\Omega\setminus \partial^\rho \Omega}\abs{\nabla I_s^\eps[u_s](x)-\int_{B_\rho(x)}\nabla_x \psi_s^\eps(\abs{x-y})u_s(y)\d y} \d x
    \\
    &\qquad+
    \int_{\partial^2\Omega\setminus \partial^\rho\Omega}
    \abs{\nabla g(x)-\int_{B_\rho(x)}\nabla_x \psi_s^\eps(\abs{x-y})g(y)\d y}\d x
    \\
    &\lesssim
    [u_s]_{C^{0,s}(\R^d)}
    \frac{1-\rho^{1-s}}{1-\eps^{2-2s}} 
    +
    \int_{\partial^2\Omega\setminus \partial^\rho\Omega}
    \abs{\nabla g(x)-\int_{B_\rho(x)}\nabla_x \psi_s^\eps(\abs{x-y})g(y)\d y}\d x.
\end{align*}
To estimate the last integral, we use that $g\in C^{1,\alpha}(\R^d)$ to write
\begin{align*}
    g(y) = g(x) + \nabla g(x)\cdot (y-x) + R_x(y),
\end{align*}
where the remainder $R_x$ satisfies 
\begin{align*}
    \abs{R_x(y)} \lesssim [g]_{C^{1,\alpha}(\R^d)} \abs{x-y}^{1+\alpha},\qquad \forall \, x,y\in\R^d.
\end{align*}
As before, we abbreviate $M(s,\eps):=\frac{2}{1-\eps^{2-2s}}$ and compute the following partial derivatives for $i=1,\dots,d$:
\begin{align*}
    \int_{B_\rho(x)}\partial_i \psi_s^\eps(\abs{x-y})&g(y)\d y
    \\
    =&
    M(s,\eps)
    \int_{B_\rho(x)\setminus B_\eps(x)}\eta_s(\abs{x-y})(x_i-y_i)(g(x)-g(y))\d y
    \\
    =&
    M(s,\eps)
    \int_{B_\rho(x)\setminus B_\eps(x)}\eta_s(\abs{x-y})(x_i-y_i)\left(\nabla g(x)\cdot(x-y) - R_x(y))\right)\d y
    \\
    =&
    M(s,\eps)
    \sum_{j=1}^d
    \partial_j g(x)
    \int_{B_\rho\setminus B_\eps}\eta_s(\abs{z})z_i z_j \d z 
    \\
    &-
    M(s,\eps)
    \int_{B_\rho(x)\setminus B_\eps(x)}\eta_s(\abs{x-y})(x_i-y_i)R_x(y)\d y.
\end{align*}
Next, we observe that the first term can be simplified as follows:
\begin{align*}
    M(s,\eps)
    \sum_{j=1}^d
    \partial_j g(x)
    \int_{B_\rho\setminus B_\eps}\eta_s(\abs{z})z_i z_j \d z
    &=
    M(s,\eps)
    \partial_i g(x)
    (1-s)
    \int_{\eps}^\rho 
    t^{1-2s}\d t
    \\
    &=
    \frac{\rho^{2-2s}-\eps^{2-2s}}{1-\eps^{2-2s}}
    \partial_i g(x)
    \\
    &=
    \partial_i g(x)
    +
    \left(\frac{\rho^{2-2s}-\eps^{2-2s}}{1-\eps^{2-2s}}-1\right)\partial_i g(x)
    \\
    &=
    \partial_i g(x)
    -
    \frac{1-\rho^{2-2s}}{1-\eps^{2-2s}}
    \partial_i g(x).
\end{align*}
The term containing the Taylor remainder $R_x(y)$ can be estimated as follows
\begin{align*}
    &\phantom{{}={}}
    \bigg|M(s,\eps)
    \int_{B_\rho(x)\setminus B_\eps(x)}\eta_s(\abs{x-y})(x_i-y_i)R_x(y)\d y\bigg|\\
    &\lesssim [g]_{C^{1,\alpha}(\R^d)}
    \frac{1-s}{1-\eps^{2-2s}}
    \int_\eps^\rho t^{1+\alpha-2s}\d t
    \\
    &\leq [g]_{C^{1,\alpha}(\R^d)}
    \frac{1-s}{1-\eps^{2-2s}}
    \frac{\rho^{2+\alpha-2s}-\eps^{2+\alpha-2s}}{\alpha}.
\end{align*}
As a consequence of these computations we obtain
\begin{align*}
    &\phantom{{}={}}\bigg|\partial_i g(x)-\int_{B_\rho(x)}\partial_i \psi_s^\eps(\abs{x-y})g(y)\d y\bigg|\\
    &
    \lesssim 
    \frac{1}{1-\eps^{2-2s}}
    \left[
    (1-\rho^{2-2s})
    \abs{\partial_i g(x)}
    +
    [g]_{C^{1,\alpha}(\R^d)}\frac{1-s}{\alpha}
    \left(\rho^{2+\alpha-2s}-\eps^{2+\alpha-2s}\right)
    \right]
\end{align*}
and hence
\begin{align*}
    I
    &\lesssim
    [u_s]_{C^{0,s}(\R^d)}
    \frac{1-\rho^{1-s}}{1-\eps^{2-2s}}
    +
    \frac{[g]_{C^{0,1}(\R^d)}(1-\rho^{2-2s})
    +
    [g]_{C^{1,\alpha}(\R^d)}\frac{1-s}{\alpha}
    \left(\rho^{2+\alpha-2s}-\eps^{2+\alpha-2s}\right)}{1-\eps^{2-2s}}
    \\
    &\lesssim
    \left([u_s]_{C^{0,s}(\R^d)}+
    \|g\|_{C^{1,\alpha}(\R^d)}\right)
    \frac{(1-\rho^{2-2s}) + \frac{1-s}{\alpha}\left(\rho^{2+\alpha-2s}-\eps^{2+\alpha-2s}\right)}{1-\eps^{2-2s}},
\end{align*}
where we used that $1-\rho^{1-s}\leq 1-\rho^{2-2s}$ for $0\leq\rho\leq 1$.
Plugging this estimate in the estimate for $G_1^3$ we get
\begin{align*}
    G_1^3 
    &\lesssim
    \left([u_s]_{C^{0,s}(\R^d)}^2+ \|g\|_{C^{1,\alpha}(\R^d)}^2\right) 
    \frac{(1-\rho^{2-2s}) + \frac{1-s}{\alpha}\left(\rho^{2+\alpha-2s}-\eps^{2+\alpha-2s}\right)}{(1-\eps^{2-2s})^2}.
\end{align*}
Finally, we estimate $G_1^4$, using that $\nabla$ and $I_s^\eps$ commute:
\begin{align*}
G_1^4&=\left|\int_{\R^d\setminus \Omega^2}|\nabla g(x)|^2\d x -  \int_{\R^d\setminus \Omega^2}|I_s^\eps[\nabla g](x)|^2\d x \right|\\
&\leq 2 \|g\|_{C^{0,1}(\R^d)} \int_{\R^d\setminus \Omega^2}|\nabla g(x) - I_s^\eps[\nabla g](x)|\d x\\
&\leq 2 \|g\|_{C^{0,1}(\R^d)} \|\nabla g - I_s^\eps[\nabla g]\|_{L^1(\R^d)}
\end{align*}
It remains to bound the last term, for which we will use that $\nabla g\in \dot W^{\beta,1}(\R^d)$ for some $\beta>0$.
Together with \cref{smoothlem-bound,smoothlem-bound_d=1} in \cref{lem:smoothlem} this allows us to estimate
\begin{align*}
    \norm{\nabla g - I_s^\eps[\nabla g]}_{L^1(\R^d)}
    &\leq 
    \int_{\R^d}
    \int_{B_1(x)}
    \psi_s^\eps(\abs{x-y})\abs{\nabla g(x)-\nabla g(y)}\d y\d x
    \\
    &\leq 
    C(1-s)
    \int_{\R^d}
    \int_{B_1(x)}
    \frac{\abs{\nabla g(x)-\nabla g(y)}}{\abs{x-y}^{d+2(s-1)}}\d y\d x
    \\
    \\
    &=
    C(1-s)
    \int_{\R^d}
    \int_{B_1(x)}
    \frac{\abs{\nabla g(x)-\nabla g(y)}}{\abs{x-y}^{d+\beta}}\underbrace{\abs{x-y}^{\beta-2(s-1)}}_{\leq 1}\d y\d x
    \\
    &\leq 
    C(1-s)[g]_{\dot W^{\beta,1}(\R^d)},
\end{align*}
where the constant $C$ just depends on the dimension $d$.
Hence, we have
\begin{align*}
    G_1^4 \leq 2C\norm{g}_{C^{0,1}(\R^d)}[g]_{\dot W^{\beta,1}(\R^d)}(1-s).
\end{align*}
In total, we arrive at the following estimate for $G_1$:
\begin{align}\label{eq:estim_G1}
\begin{split}
    G_1 
    &\lesssim
    \left([u_s]^2_{C^{0,s}(\R^d)}+\|g\|_{C^{1,\alpha}(\R^d)}^2+
    \norm{g}_{C^{0,1}(\R^d)}[g]_{\dot W^{\beta,1}(\R^d)}\right)
    \times
    \\
    &\qquad\qquad
    \times\left[
    \rho +
    \frac{\rho + (1-\rho^{2-2s}) + \frac{1-s}{\alpha}\left(\rho^{2+\alpha-2s}-\eps^{2+\alpha-2s}\right)}{(1-\eps^{2-2s})^2}
    +1-s
    \right].
\end{split}
\end{align}
Summing \labelcref{eq:estim_G1,eq:estim_G2} and keeping only the dominating terms, we obtain the full estimate of the difference of the Dirichlet energies:
\begin{align*}
    |D(w_s^{r,\eps}) &- D(I_s^\eps[u_s])|
    \lesssim
    \left([u_s]_{C^{0,s}(\R^d)}^2 + \|g\|_{C^{1,\alpha}(\R^d)}^2+
    \norm{g}_{C^{0,1}(\R^d)}[g]_{\dot W^{\beta,1}(\R^d)}\right) \times
    \\
    &\qquad
      \times \Bigg[ r^{2s-1}  + 
    \frac{(r+\rho)+ (1-\rho^{2-2s}) + \frac{1-s}{\alpha}\left(\rho^{2+\alpha-2s}-\eps^{2+\alpha-2s}\right) + (1-s)^2r^{-1}}{(1-\eps^{2-2s})^2}
    \Bigg].
\end{align*}
We prove now the analogue estimate when $g\equiv 0$ (i.e. \cref{lem-suitgrad-item3}). We can improve the estimate of $G_1^3$ as follows:
\begin{align*}
    G_1^3 
    = \int_{\partial^1\Omega\setminus\partial_\rho\Omega}\abs{\nabla I_s^\eps[u_s]}^2\d x
    \lesssim [u_s]_{C^{0,s}(\R^d)}
    \left(\frac{1-\rho^{1-s}}{1-\eps^{2-2s}}\right)^2,
\end{align*}
where we used \labelcref{eq:grad_tail_conv_holder_s} together with the fact that
\begin{align*}
    \int_{B_\rho(x)}\nabla_x \psi_s^\eps(\abs{x-y})u_s(y)\d y
    =
    \int_{B_\rho(x)}\nabla_x \psi_s^\eps(\abs{x-y})g(y)\d y
    =0
\end{align*}
for all $x\in\partial^1\Omega\setminus\partial^\rho\Omega$.
In this case we get the better estimate
\begin{align*}
    \abs{D(w_s^{r,\eps}) - D(I_s^\eps[u_s])} 
    &\lesssim
    [u_s]_{C^{0,s}(\R^d)}^2
    \Bigg[r^{2s-1}  
    +
    \frac{(r+\rho)+(1-\rho^{1-s})^2 + (1-s)^2r^{-1}}{(1-\eps^{2-2s})^2}
    \Bigg].
\end{align*}
Finally, it remains to bound the difference of the linear terms in the energy $J$, taking into account that
\begin{align*}
    \abs{J(w_s^{r,\eps}) - J(I_s^\eps[u_s])}
    \leq 
    \abs{D(w_s^{r,\eps}) - D(I_s^\eps[u_s])}
    +
    \int_\Omega 
    \abs{f(x)}
    \abs{w_s^{r,\eps}(x) - I_s^\eps[u_s](x)}
    \d x.
\end{align*}
Using \cref{lem:suit}, the second term can be estimated as follows:
\begin{align*}
    \int_\Omega 
    \abs{f(x)}&
    \abs{w_s^{r,\eps}(x) - I_s^\eps[u_s](x)}
    \d x\\
    &\leq 
    \norm{f}_{L^\infty(\Omega)}
    \int_{\partial_r\Omega}
    \abs{w_s^{r,\eps}(x) - I_s^\eps[u_s](x)}
    \d x
    \\
    &\lesssim 
    \norm{f}_{L^\infty(\Omega)}
    \left([u_s]_{C^{0,s}(\R^d)} + [g]_{C^{0,s}(\R^d)}\right)\left(r^{1+s}+
    \frac{1-s}{1-\eps^{2-2s}}r\right).
\end{align*}
\end{proof}

\subsection{Proof of the main results}
\label{sec:proof_main}

\begin{proof}[Proofs of \cref{thm:main,thm:less_regular_bdry}]

Thanks to \cite[Proposition 1.1]{ros-oton2014dirichlet} we have $u_s \in C^{0,s}(\R^d)\cap W^{s,2}(\R^d)$ in the case $g\equiv 0$.
In the setting of \cref{thm:less_regular_bdry} we pose this as an assumption.
\vspace{2mm}

\textbf{Step 1:} Let us estimate first the difference of $u$ and $u_s$ in the $W^{s,2}(\R^d)$-seminorm. We consider the function $I_s^\eps[u_s]$ which by \cref{prop:nonlocal2local} lies in $W^{1,2}(\R^d)$ and we have control of its Dirichlet energy.
However, $I_s^\eps[u_s]$ is not feasible for the nonlocal problem \labelcref{eq:nonlocal_problem} since the mollification through $I_s^\eps$ leads to $I_s^\eps[u_s]\neq g$ on $\R^d\setminus\Omega$.
Therefore, we consider the feasible function $w_s^{r,\eps} \in W^{1,2}(\R^d)$ given by \labelcref{eq:boundmol}. 

In the following estimate we use, in order, the nonlocal stability result in \cref{prop:convexity_nonlocal}, the local to nonlocal upper consistency of the energies in \cref{prop:consistency}, as well as feasibility of $w_s^{r,\eps}$ for the local minimization problem (that is minimized by $u$):
\begin{align*}
    \frac{1}{2} [u-u_s]_{W^{s,2}(\R^d)}^2
    &\leq J_s(u)-J_s(u_s)
    \\
    &\leq J(u) - J_s(u_s) + \frac{2d}{s} \|u\|_{L^2(\R^d)}^2 (1-s)
    +\norm{u}_{L^\infty(\Omega)}\norm{f-f_s}_{L^1(\Omega)}
    \\
    &\leq
    J(w^{r,\eps}_s) - J_s(u_s) + \frac{2d}{s} \|u\|_{L^2(\R^d)}^2 (1-s)
    +\norm{u}_{L^\infty(\Omega)}\norm{f-f_s}_{L^1(\Omega)}.
\end{align*}
The last two terms in the last line are already of the required form for our result. Let us estimate the first one. First we write
\begin{equation}\label{eq:E1E2}
\begin{split}
   J(w^{r,\eps}_s) - J_s(u_s)
   &= 
   \big(J(I_s^\eps[u_s]) - J_s(u_s)\big) + \big(J(w^{r,\eps}_s))-J(I_s^\eps[u_s])\big)\\
   &=: E_1(s,\eps) + E_2(s,\eps,r,\rho).
   \end{split}
\end{equation}
To bound $E_1$ we use the definition of $J_s$, the positivity of the non-singular part of $J_s$ given by $D^2_s$, the adaptability of $I_s^\eps$ with respect $J$ and $J_s$ given in \cref{prop:nonlocal2local}, and the error estimate of the mollifier in \cref{prop:convolution_pointwise}  to get
\begin{align*}
    E_1(s,\eps) 
    &= 
    \left(D(I_s^\eps[u_s])- \int_{\Omega} f(x)I_s^\eps[u_s](x) \d x\right)- \left(D^1_s[u_s] + D^2_s[u_s]-\int_{\Omega} f_s(x)u_s(x) \d x\right)
    \\
    &\leq \left(D(I_s^\eps[u_s])- D^1_s[u_s] \right) + \left(\int_{\Omega} f_s(x)u_s(x) \d x - \int_{\Omega} f(x)I_s^\eps[u_s](x) \d x \right)
    \\
    &\leq \left(\frac{1}{(1-\eps^{2-2s})^2}-1\right)D^1_s[u_s] 
    \\
    &\qquad+ \left(\int_{\Omega} (f_s(x)-f(x))u_s(x) \d x + \int_{\Omega} f(x)(u_s(x)-I_s^\eps[u_s](x)) \d x \right)
    \\
    &\leq
    \left(\frac{1}{(1-\eps^{2-2s})^2}-1\right)D^1_s[u_s] 
    \\
    &\qquad+ \|u_s\|_{L^\infty(\R^d)}\|f_s-f\|_{L^1(\Omega)} + \|f\|_{L^1(\Omega)}\|u_s-I_s^\eps[u_s]\|_{L^\infty(\R^d)}
    \\
    &\leq \left(\frac{1}{(1-\eps^{2-2s})^2}-1\right)D^1_s[u_s] 
    \\
    &\qquad+ \|u_s\|_{L^\infty(\R^d)}\|f_s-f\|_{L^1(\Omega)} + \|f\|_{L^1(\Omega)} 2\frac{1-s}{1-\eps^{2-2s}}[u_s]_{C^{0,s}(\R^d)}.
\end{align*}
At this point, we can take limits as $\eps \to 0$ to get
\begin{align*}
    \frac{1}{2} [u-u_s]_{W^{s,2}(\R^d)}^2
    &\leq
    \liminf_{\eps\to0} E_1(s,\eps)+ \liminf_{\eps\to0} E_2(s,\eps,r,\rho) + \frac{2d}{s} \|u\|_{L^2(\R^d)}^2 (1-s)
    \\
    &\qquad
    +\norm{u}_{L^\infty(\Omega)}\norm{f-f_s}_{L^1(\Omega)}
    \\
    &\lesssim    \left(\|u\|_{L^\infty(\R^d)}+\|u_s\|_{L^\infty(\R^d)}\right)\|f_s-f\|_{L^1(\Omega)} + (1-s)\|f\|_{L^1(\Omega)} [u_s]_{C^{0,s}(\R^d)}\\
    &\qquad
    + \liminf_{\eps\to0}E_2(s,\eps,r,\rho).
\end{align*}
The bound for $E_2$ is precisely the one given in \cref{lem:lastlemma} and is a continuous function of $0<\eps<1$. 
Since the above estimate already has an error term of order $1-s$, we will optimize the parameters $r$ and $\rho$ to match this error. 
Let us do it first in the case when $g\equiv0$. 
For clarity of the presentation we avoid all the dependence of the norms of $g$, $u_s$ and $f$ and also the dimensional and domain constants. In this way, we get
\begin{align*}
    \liminf_{\eps\to0}E_2(s,\eps,r,\rho)
    &\lesssim 
    r^{2s-1} 
    + 
    r+\rho+ (1-\rho^{1-s})^2  + (1-s)^2r^{-1}+
    r^{1+s}
    +
    (1-s)r\\
    &\lesssim \rho+ (1-\rho^{1-s})^2  + r^{2s-1} +(1-s)^2r^{-1}+
    (1-s),
\end{align*}
where we only kept the dominating terms in $r$ for $r\in(0,1)$. Now let us choose $r=(1-s)^{\frac{1}{s}}$ to get
\[
r^{2s-1}=(1-s)^2r^{-1}= (1-s)(1-s)^{1-\frac{1}{s}}\leq e(1-s),
\]
and $\rho=1-s$ to get
\[
(1-\rho^{1-s})^2=(1-(1-s)^{1-s})^2  \leq (1-s).
\]
This concludes the proof when $g\equiv0$. We do it now in the case $g\in C^{1,\alpha}(\R^d)$ for some $\alpha>0$. This time we have, using again the choice $r=(1-s)^{\frac{1}{s}}$, and the fact that $\rho\leq1$, the following bound
\[
E_2(s,0,r,\rho)\lesssim\rho+(1-\rho^{2-2s})+\left(1+\frac{1}{\alpha}\right)(1-s).
\]
For the choice $\rho=(1-s)\abs{\log\left(\frac{1-s}{2}\right)}$ one can easily show that $1-\rho^{2-2s}\leq2\rho$, which concludes the proof of the estimate of $[u-u_s]_{W^{s,2}(\R^d)}$.

\vspace{2mm}
\textbf{Step 2:} Let us bound now the difference of $u$ and $u_s$ in $L^2(\Omega)$. First, by triangle inequality, 
\begin{align}\label{eq:L2error}
    \norm{u-u_s}_{L^{2}(\Omega)}& \leq \norm{u-w_s^{r,\eps}}_{L^{2}(\Omega)}+ \norm{w_s^{r,\eps}-I_s^\eps[u_s]}_{L^{2}(\Omega)}+ \norm{I_s^\eps[u_s]-u_s}_{L^{2}(\Omega)}.
\end{align}
The last two terms can be bounded by \cref{cor:L2_error_modification,prop:convolution_pointwise} as follows
\begin{align*}
    \norm{w_s^{r,\eps}-I_s^\eps[u_s]}_{L^{2}(\Omega)}^2&\leq C_\Omega\left([u_s]_{C^{0,s}(\R^d)}^2 + [g]_{C^{0,s}(\R^d)}^2\right)  \left( r^{1+2s} +\left(\frac{1-s}{1-\eps^{2-2s}}\right)^2 r\right), \\
    \norm{I_s^\eps[u_s]-u_s}_{L^{2}(\Omega)}^2&\leq C_{\Omega}[u_s]_{C^{0,s}(\R^d)}^2\left(\frac{1-s}{1-\eps^{2-2s}}\right)^2 .
\end{align*}
Let us bound now the first term on the right hand side of \labelcref{eq:L2error}. By Poincaré's inequality and \cref{prop:convexity_local}, we get
\begin{align*}
    \frac{1}{C_\Omega^2}\norm{u-w_s^{r,\eps}}_{L^{2}(\Omega)}^2 \leq   \norm{\nabla u-\nabla w_s^{r,\eps}}_{L^{2}(\Omega)}^2=J(w_s^{r,\eps})- J(u).
\end{align*}
Thus, we can use the local to nonlocal upper consistency of the energies in \cref{prop:consistency}, together with the fact that $u$ is feasible for the nonlocal minimization problem (that is minimized by $u_s$), to get
\begin{align*}
    &\phantom{{}={}}
     \frac{1}{C_\Omega^2}\|u-w_s^{r,\eps}\|_{L^{2}(\Omega)}^2 
     \\
     &\leq  J(w_s^{r,\eps}) - J(I_s^\eps[u_s])+ J(I_s^\eps[u_s])- J(u)\\
     &\leq E_2(s,\eps,r,\rho) + J(I_s^\eps[u_s])- J_s(u) + \frac{2d}{s} \|u\|_{L^2(\R^d)}^2 (1-s)
    +\norm{u}_{L^\infty(\Omega)}\norm{f-f_s}_{L^1(\Omega)}\\
    &\leq E_2(s,\eps,r,\rho) + J(I_s^\eps[u_s])- J_s(u_s) + \frac{2d}{s} \|u\|_{L^2(\R^d)}^2 (1-s)
    +\norm{u}_{L^\infty(\Omega)}\norm{f-f_s}_{L^1(\Omega)}\\
    &=E_2(s,\eps,r,\rho) +E_1(s,\eps) + \frac{2d}{s} \|u\|_{L^2(\R^d)}^2 (1-s)
    +\norm{u}_{L^\infty(\Omega)}\norm{f-f_s}_{L^1(\Omega)},
\end{align*}
where $E_1$ and $E_2$ are the ones defined in \labelcref{eq:E1E2}. From here, we can argue as in Step 1 to obtain the same orders of convergence obtained there.

\end{proof}

\begin{proof}[Proof of \cref{cor:main_regular_bdry}]
    Define $\tilde u_s := u_s - g$ and $\tilde u := u - g$ which satisfy $\tilde u_s, u \equiv 0$ in $\R^d\setminus\Omega$ and are weak solutions of
    \begin{align*}
        (-\Delta)^s\tilde u_s = \tilde f_s\quad\text{in }\Omega, \quad \textup{and} \quad -\Delta\tilde u = \tilde f\quad\text{in }\Omega,
    \end{align*}
    where $\tilde f_s := f_s - (-\Delta)^s g$ and $\tilde f := f + \Delta g$.
    Since $g\in C^{2,\alpha}(\R^d)$ for $\alpha>0$ we have that $\abs{-(-\Delta)^s g - \Delta g} \lesssim \alpha^{-1}(1-s)$ in $\Omega$ by \cref{lem:pointwise_consistency} in \cref{sec:pointwise_consistency} and, as a consequence,
    \begin{align*}
        \norm{\tilde f_s - \tilde f}_{L^p(\Omega)} \leq \norm{f_s - f}_{L^p(\Omega)} + \norm{-(-\Delta)^s g - \Delta g}_{L^p(\Omega)}
        \lesssim
        \norm{f_s - f}_{L^p(\Omega)}
        +
        \alpha^{-1}(1-s).
    \end{align*}
    Hence, applying \cref{thm:main} to $\tilde u_s$ and $\tilde u$ lets us conclude.
\end{proof}


\section*{Acknowledgement} FdT was supported by the Spanish Government through RYC2020-029589-I, PID2021-127105NB-I00 and CEX2019-000904-S funded by the MICIN/AEI.
The idea for this project originated from first discussions while the authors were in residence at Institut Mittag-Leffler in Djursholm, Sweden during the semester on \textit{Geometric Aspects of Nonlinear Partial Differential Equations} in 2022, supported by the Swedish Research Council under grant no. 2016-06596.
The authors would also like to thank Erik Lindgren for discussion on the subject of this paper.

\printbibliography[heading=bibintoc] 

\begin{appendix}    
    \section{The normalization constant}
    \label{sec:normalization}

    The choice of $C_{s,d}$ is the natural one for the computations in this paper and leads to pretty clean estimates. 
    However, if one wants to consider the classical constant arising from the Fourier symbol of the fractional Laplacian, i.e.
    \[
    \tilde{C}_{s,d}=\frac{4^{s-1} \Gamma\left(s+\frac{d}{2}\right)}{\pi^{\frac{d}{2}}\Gamma(2-s)}s(1-s),
    \]
    it is standard to check, using the identity $\omega_d=(2\pi^{d/2})/\Gamma(d/2)$ and the fundamental property of the Gamma-function, that
    \begin{align*}
        \frac{C_{s,d}}{\tilde{C}_{s,d}}= \frac{\frac{d }{\omega_d}(1-s)}{\frac{4^{s-1} \Gamma\left(s+\frac{d}{2}\right)}{\pi^{\frac{d}{2}}\Gamma(2-s)}s(1-s)}
        =
        \frac{\Gamma\left(1+\frac{d}{2}\right)}{\Gamma\left(s+\frac{d}{2}\right)}
        \frac{\Gamma(2-s)}{4^{s-1}}.
    \end{align*}
    Combining this with convexity of the Gamma-function and Gautschi's inequality one can easily show that 
    \begin{align*}
        \abs{1-\frac{\tilde C_{s,d}}{C_{s,d}}} \leq C(1-s)
    \end{align*}
    for a universal constant $C>0$.
    Hence, if one is interested in solutions to the fractional Laplace equation with the classical normalization, one can simply absorb the quotient of the two constants into the right hand side $f_s$ which leads to a right hand side $\tilde f_s$ with 
    \[
    \norm{f_s-\tilde f_s}_{L^1(\Omega)}=\abs{1-\frac{\tilde C_{s,d}}{C_{s,d}}}\norm{f_s}_{L^1(\Omega)}\leq C(1-s)\norm{f_s}_{L^1(\Omega)}.\]
    As we saw in our main result, the order of convergence is hence not affected by the choice of $C_{s,d}$ instead of $\tilde C_{s,d}$.
    
    \section{Pointwise consistency for regular functions}
    \label{sec:pointwise_consistency}
    
    Our arguments for proving \cref{thm:main} are purely variational, however, for \cref{cor:main_regular_bdry} we needed the following pointwise consistency statement for the fractional Laplacian of a sufficiently smooth function. 
    \begin{lemma}\label{lem:pointwise_consistency}
        There exists a constant $C>0$ such that for any $g\in C^{2,\alpha}(\R^d)$ with $\alpha>0$, any $x\in\R^d$, and any $s\in(0,1)$ it holds that
        \begin{align*}
           |-\Delta g(x)- (-\Delta)^s g(x)| = C\left([g]_{C^{2,\alpha}(\R^d)}+\|g\|_{L^\infty(\R^d)}\right)\frac{1-s}{s}.
        \end{align*}
    \end{lemma}
    \begin{proof}
    By definition,
    \begin{align*}
         (-\Delta)^s g(x)=&\ 4 C_{d,s} \operatorname{P.V.} \int_{|z|>0} \left(g(x) - g(x+z)\right) \frac{\d z}{|z|^{d+2s}}\\
         =&\ 4C_{d,s}\operatorname{P.V.} \int_{0<|z|<1} \left(g(x) - g(x+z)\right) \frac{\d z}{|z|^{d+2s}}\\
         &+4C_{d,s}\int_{|z|>1} \left(g(x) - g(x+z)\right) \frac{\d z}{|z|^{d+2s}}\\
         =&\ I_1(x)+I_2(x).
    \end{align*}
Clearly, 
\begin{align*}
    |I_2(x)| &\leq 8 C_{d,s}\|g\|_{L^\infty(\R^d)} \int_{|z|>1} \frac{\d z}{|z|^{d+2s}} =
    4 \|g\|_{L^\infty(\R^d)} \frac{d}{s} (1-s).
\end{align*}
For $I_1$ we note that, by Taylor expansion,
\[
\left|g(x+z)-g(x)- \nabla g(x) \cdot z - \frac{1}{2} z^T D^2 g(x) z\right| \leq K [g]_{C^{2,\alpha}(\R^d)} |z|^{2+\alpha}.
\]
Thus, by symmetry, 
\begin{align*}
\left|I_1(x)- \frac{4 C_{d,s}}{2} \int_{0<|z|<1} (z^T D^2 g(x) z )\frac{\d z}{|z|^{d+2s}} \right| &\leq 4 K [g]_{C^{2,\alpha}(\R^d)}C_{s,d} \omega_d\int_{0}^1 t^{1+\alpha-2s} \d t\\
&= 4K [g]_{C^{2,\alpha}(\R^d)} \frac{d}{2-2s+\alpha} (1-s). 
\end{align*}
Finally note that, again by symmetry, we have
\begin{align*}
\frac{4 C_{d,s}}{2} \int_{0<|z|<1} (z^T D^2 g(x) z )\frac{\d z}{|z|^{d+2s}}&= 2 C_{d,s}\sum_{i=1}^d \frac{\partial^2 g}{\partial x_i^2} (x) \int_{0<|z|<1} z_i^2\frac{\d z}{|z|^{d+2s}}\\
&= \frac{2 C_{d,s}}{d} \sum_{i=1}^d \frac{\partial^2 g}{\partial x_i^2} (x) \int_{0<|z|<1} |z|^2\frac{\d z}{|z|^{d+2s}}\\
&= \frac{2 C_{d,s}}{d} \left(\omega_d 
 \int_0^1 t^{1-2s} \d t\right) \sum_{i=1}^d \frac{\partial^2 g}{\partial x_i^2} (x) 
\\
&= \Delta g(x),
\end{align*}
which concludes the proof.
\end{proof}

\section{Stability of the Dirichlet problems}
\label{sec:stability}

\begin{proof}[Proof of \cref{prop:convexity_local}]
Expanding the square, one has the following vectorial identity:
\begin{align*}
    \abs{a-b}^2 + \abs{a}^2 + 2\langle a,b-a\rangle = \abs{b}^2,\qquad\forall a,b\in\R^d.
\end{align*}
Hence, it holds
\begin{align*}
    \int_\Omega \abs{\grad \phi(x) - \grad u(x)}^2\d x
    +
    2\int_\Omega 
    \langle
    \grad u(x), \grad \phi(x) -& \grad u(x)
    \rangle 
    \d x\\
    &=
    \int_\Omega 
    \abs{\grad \phi(x)}^2 \dx 
    -
    \int_\Omega 
    \abs{\grad u(x)}^2 \dx.
\end{align*}
Since $u$ is a minimizer of \labelcref{eq:local_problem}, it follows
\begin{align*}
    \int_\Omega 
    \langle
    \grad u(x), \grad w(x)
    \rangle 
    \d x
    =
    \int_\Omega 
    f(x) w(x) \d x,\qquad\forall w\in W^{1,2}_0(\Omega).
\end{align*}
Using this for $w := \phi-u \in W^{1,2}_0(\Omega)$, we obtain
\begin{align*}
    \int_\Omega \abs{\grad \phi(x) - \grad u(x)}^2\d x +
    2
    \int_\Omega 
    f(x) (\phi(x)-u(x)) \d x
    = 
    \int_\Omega 
    \abs{\grad \phi(x)}^2 \dx 
    -
    \int_\Omega 
    \abs{\grad u(x)}^2 \dx.
\end{align*}
Finally, since $\phi=u$ in $\R^d\setminus\Omega$, We can reorder the above identity and get
\begin{align*}
    \frac{1}{2}\norm{\nabla\phi-\nabla u}_{L^2(\R^d)}^2
    = 
    J(\phi) - J(u).
\end{align*}
\end{proof}
    \begin{proof}[Proof of \cref{prop:convexity_nonlocal}]
    Expanding the square one has
    \begin{align*}
        (a-b)^2 + a^2 + 2a(b-a) = b^2,\qquad\forall a,b\in\R.
    \end{align*}
    Taking $a=u_s(x)-u_s(y)$ and $b=\phi(x)-\phi(y)$ in the above identity, it holds that
    \begin{align*}
        &\frac{1}{2}[\phi - u_s]_{W^{s,2}(\R^d)}^2=
        \int_{\R^d}
        \int_{\R^d}
        \eta_s(\abs{x-y})
        \abs{(\phi(x)-u_s(x)) - (\phi(y)-u_s(y))}^2
        \d y\d x\\
        & = 
        \int_{\R^d}
        \int_{\R^d}
        \eta_s(\abs{x-y})
        \abs{(\phi(x)-\phi(y)) - (u_s(x)-u_s(y))}^2
        \d y\d x
        \\
        & = 
        -2
        \int_{\R^d}
        \int_{\R^d}
        \eta_s(\abs{x-y})
        (u_s(x)-u_s(y))((\phi(x)-\phi(y))-(u_s(x)-u_s(y)))\d y \d x
        \\
        &\quad 
        + 
        \int_{\R^d}
        \int_{\R^d}
        \eta_s(\abs{x-y})
        \abs{\phi(x)-\phi(y)}^2
        \d y \d x
        -
        \int_{\R^d}
        \int_{\R^d}
        \eta_s(\abs{x-y})
        \abs{u_s(x)-u_s(y)}^2
        \d y \d x.\\
        & = 
        -2
        \int_{\R^d}
        \int_{\R^d}
        \eta_s(\abs{x-y})
        (u_s(x)-u_s(y))((\phi(x)-u_s(x))-(\phi(y)-u_s(y)))\d y \d x
        \\
        &\quad 
        + 
        J_s(\phi) -J_s(u_s) + \int_{\Omega} f_s(x)(\phi(x)-u_s(x))\d x .
    \end{align*}
    Now note that, since $u_s$ is a minimizer of \labelcref{eq:nonlocal_problem}, it follows
    \begin{align*}
            2\int_{\R^d} \int_{\R^d}\eta_s(\abs{x-y})
        (u_s(x)-u_s(y))(w(x)-w(y))\d y \d x
        =
        \int_\Omega 
        f_s w \d x,\qquad\forall w\in W^{s,2}_0(\Omega).
    \end{align*}
    Using this for $w := \phi-u_s \in W^{s,2}_0(\Omega)$ we obtain the desired result.
    \end{proof}

\end{appendix}

\end{document}